\documentclass[12pt,svgnames]{amsart}

\usepackage{amsmath,amsthm,amssymb,amscd, amsfonts,enumerate,array,bbm,bm,xypic}
\usepackage{hyperref}
\usepackage{tikz}
\usepackage{verbatim}
\usepackage{young}
\usepackage{xcolor}

\ysetshade{Yellow!50}
\YSetShade{Green!50}
\ysetaltshade{Red!50}
\YSetAltShade{Blue!50}
\YSETSHADE{Orange!50}
\YSETALTSHADE{Purple!50}

\allowdisplaybreaks[2]
\makeatletter
\def\mhline{\noalign{\ifnum0=`}\fi\hrule height 4\arrayrulewidth \futurelet
   \@tempa\oxhline}
\def\oxhline{\ifx\@tempa\hline\vskip \doublerulesep\fi
      \ifnum0=`{\fi}}

\makeatother

\numberwithin{equation}{section}

\newtheorem{theorem}{Theorem}[section]
\newtheorem{proposition}[theorem]{Proposition}

\newtheorem{problem}[theorem]{Problem}
\newtheorem{corollary}[theorem]{Corollary}
\newtheorem{lemma}[theorem]{Lemma}

\theoremstyle{definition}

\newtheorem{remark}[theorem]{Remark}
\newtheorem{example}[theorem]{Example}
\newtheorem{definition}[theorem]{Definition}

\newcommand{\lan}{\langle}
\newcommand{\ran}{\rangle}
\renewcommand{\eqref}[1]{{\rm (\ref{#1})}}
\DeclareMathOperator{\TFF}{TFF}

\def\ZZ{\mathbb{Z}}

\def\RR{\mathbb{R}}

\def\NN{\mathbb{N}}

\def\rank{\operatorname{rank}}

\begin{document}

\title{A combinatorial characterization of tight fusion frames}

\author{Marcin Bownik}
\address{\noindent Department of Mathematics, University of Oregon,
Eugene, OR 97403, USA}
\email{mbownik@uoregon.edu}

\author{Kurt Luoto}
\address{Department of Mathematics, University of British Columbia, Vancouver, BC V6T 1Z2, Canada}
\email{kwluoto@math.ubc.ca}

\author{Edward Richmond}
\address{Department of Mathematics, University of British Columbia, Vancouver, BC V6T 1Z2, Canada}
\email{erichmond@math.ubc.ca}

\begin{abstract}
In this paper we give a combinatorial characterization of tight fusion frame (TFF) sequences using
Littlewood-Richardson skew tableaux. The equal rank case has been solved recently by Casazza et al.
\cite{CFMUZ}. Our characterization does not have this limitation.  We also develop some methods for
generating TFF sequences.  The basic technique is a majorization principle for TFF sequences
combined with spatial and Naimark dualities.  We use these methods and our characterization to give
necessary and sufficient conditions which are satisfied by the first three highest ranks. We also
give a combinatorial interpretation of spatial and Naimark dualities in terms of
Littlewood-Richardson coefficients. We exhibit four classes of TFF sequences which have unique
maximal elements with respect to majorization partial order. Finally, we give several examples
illustrating our techniques including an example of tight fusion frame which can not be constructed
by the existing spectral tetris techniques \cite{CCHKP, CFHWZ, CFMUZ}. We end the paper by giving a
complete list of maximal TFF sequences in dimensions $\le 9$.
\end{abstract}



\date{\today}

\maketitle


\section{Introduction} \label{S1}

Fusion frames were introduced by Casazza, Kutyniok in \cite{CK04} (under the name frames of
subspaces) and \cite{CKL08}. A fusion frame for $\mathbb R^N$ is a finite collection of subspaces
$\{W_i\}_{i=1}^K$ in $\mathbb R^N$ such that there exists constants $0<\alpha \le \alpha' <\infty$
satisfying
\[
\alpha ||x||^2 \le \sum_{i=1}^K ||P_i x||^2 \le \alpha' ||x||^2
\qquad\text{for all }x\in \RR^N,
\]
where $P_i$ is the orthogonal projection onto $W_i$. Equivalently, $\{W_i\}_{i=1}^K$ is a fusion
frame if and only if
\[
\alpha \mathbf I \le \sum_{i=1}^K P_i \le \alpha' \mathbf I,
\]
where $\mathbf I$ is the identity on $\mathbb R^N$. The constants $\alpha$ and $\alpha'$ are called
fusion frame bounds. An important class of fusion frames are {\it tight fusion frames} (TFF), for
which $\alpha=\alpha'$ and hence $ \sum_{i=1}^K P_i = \alpha \mathbf I$. We note that the
definition of fusion frames given in \cite{CK04, CKL08} applies to closed subspaces in any Hilbert
space together with a collection of weights associated to each subspace $W_i$. Since the scope of
this paper is limited to non-weighted finite dimensional TFF, the definition of a fusion frame is
only presented for this case.

Fusion frames have been a very active area of research in the frame theory.  A lot of effort was
devoted into developing the basic properties and constructing fusion frames with desired
properties. In particular, the construction and existence of sparse tight fusion frames was studied
in \cite{CCHKP}. Fusion frame potentials have been studied in \cite{CF09} and \cite{MRS10}.
Applications of fusion frames include sensor networks \cite{CKL08}, coding theory \cite{Bod07,
KPCL09}, compressed sensing \cite{BKR}, and filter banks \cite{CFM}. In this paper we consider a
problem of classifying TFF sequences.

\begin{problem}\label{problem}
Given $N\in \NN$, characterize sequences $(L_1,\ldots, L_K)$ for which there exists a tight fusion
frame $\{W_i\}_{i=1}^K$ with $\dim W_i=L_i$ in $N$ dimensional space. Equivalently, given $\alpha
> 1$ such that $\alpha N \in \NN$, characterize sequences $(L_1,\ldots, L_K)$ such that $\alpha
\mathbf I$ can be decomposed as a sum of projections $P_1+\ldots+P_K$ with $\operatorname{rank}
P_i=L_i$, $i=1,\ldots, K$.
\end{problem}

Casazza, Fickus, Mixon, Wang, and Zhou \cite{CFMUZ} have recently achieved significant progress in
this direction by solving the equal rank case. That is, the authors have classified all triples
$(K,L,N)$ such that there exists a tight fusion frame consisting of $K$ subspaces $\{W_i\}_{i=1}^K$
with the same dimension $\dim W_i=L$ in $\mathbb R^N$. The answer given in \cite{CFMUZ} is highly
non-trivial in the most interesting case when $L$ does not divide $N$ and $2L<N$. The authors show
that a necessary condition for such sequences $(K,L,N)$ is that $K \ge \lceil N/L \rceil +1$,
whereas a sufficient condition is  $K \ge \lceil N/L \rceil +2$. In a gray area, where $K = \lceil
N/L \rceil + 1$, the authors have devised a reduction procedure which replaces the original
sequence by another one with the equivalent TFF property (existence or non-existence). Then, it is
shown that after a finite number of steps the original sequence $(K,L,N)$ is reduced to one for
which either the necessary condition fails or the sufficient condition holds. However, the results
of \cite{CFMUZ} do not say much about a more general problem of classifying TFF sequences with
non-equal ranks. In this paper we answer Problem \ref{problem} by giving a combinatorial
characterization of TFF sequences using Littlewood-Richardson skew tableaux.

While the concept of fusion frames is relatively new, the problem of representing an operator as a
sum of orthogonal projections has been studied for a long time in the operator theory. The first
fundamental result of this kind belongs to Fillmore \cite{F69} who characterized finite rank
operators which are finite sums of projections, see Theorem \ref{fill}. Fong and Murphy \cite{FM}
characterized operators which are positive combinations of projections. Analogous results were
recently investigated for C-$*$ algebras and von Neumann algebras, see \cite{HKNZ, KNZ}. However,
the most relevant results for us are due to Kruglyak, Rabanovich, and Samo{\u\i}lenko \cite{KRS02,
KRS03} who characterized the set of all $(\alpha,N)$ such that $\alpha \mathbf I$ is the sum of $K$
orthogonal projections. In other words, their main result \cite[Theorem 7]{KRS03} gives a minimal
length $K$ of a TFF sequence in $\RR^N$ with the frame bound $\alpha$. However, \cite{KRS03} does
not say anything about the ranks of projections which is a focus of this paper.

In the finite dimensional setting the existence of TFF sequences is intimately related to Horn's
problem \cite{Ho62} which has been solved by Klyachko \cite{Kly98}, and Knutson and Tao \cite{KT99,
KTW04}, for a survey see \cite{Fu00, KT01}. Problem \ref{problem} can be thought of as a very
special kind of Horn's problem where hermitian matrices have only two eigenvalues: $0$ and $1$, and
their sum has only one eigenvalue $\alpha$. Using Klyachko's result \cite{Kly98} we show that the
existence of TFF sequence $(L_1,\ldots,L_K)$ is equivalent to the non-vanishing of a certain
Littlewood-Richardson coefficient, see Theorem \ref{th:combchar}. In turn, the latter condition is
equivalent to the existence of a matrix satisfying  some computationally explicit properties such
as: constant row and column sums, and row and column sum dominance, see Corollary
\ref{th:combcharcor}. Our combinatorial characterization enables us to deduce several properties
that TFF sequences must satisfy. In addition, it enables us to give an explicit construction
procedure of a tight fusion frame corresponding to a given TFF sequence, see Example \ref{ex1a}.

A fundamental technique of our paper is a majorization principle involving the majorization partial
order $\preccurlyeq$ as in the Schur-Horn theorem \cite{AMRS07, KW10}, which is also known as the
dominance order in algebraic combinatorics \cite{Fulton97}. In Section \ref{S2} we show that a
sequence majorized by a TFF sequence is also a TFF sequence. We also establish the spatial and
Naimark dualities for general TFF sequences extending the equal rank results in \cite{CFMUZ}. In
Section \ref{S4} we find necessary and sufficient conditions on the first three largest ranks of
projections using Filmore's theorem \cite{F69} and a description of possible spectra of a sum of
two projections, see Lemma \ref{sum2}. The latter result might be of independent interest since its
proof uses honeycomb models developed by Knutson and Tao \cite{KT99, KT01}. In the same section we
also exhibit classes of TFF sequences which have only one maximal element. These include not only
the expected case of integer $\alpha$, but also half-integer scenario, and the corresponding
conjugate $\alpha$'s via the Naimark duality. In Section \ref{S3} we prove our main
characterization result of TFF sequences using Littlewood-Richardson skew tableaux. In addition to
illustrating it on specific examples, in Section \ref{S5} we give a complete proof of Theorem
\ref{re} using the combinatorics of the Schur functions. This leads to a partial characterization
of TFF sequences which are of the hook type, i.e., sequences ending in repeated $1$'s. In Section
\ref{S6} we show that the spatial and Naimark dualities manifest themselves as identities for the
corresponding Littlewood-Richardson coefficients. In the final Section \ref{S7} we give several
examples of existence of tight fusion frames using skew Littlewood-Richardson tableaux. In
particular, we give an explicit construction of TFF corresponding to the sequence $(4,2,2,2,1)$ in
dimension $N=6$. This example is remarkable for two reasons. It is the first TFF sequence which is
missed by brute force generation involving recursive spatial and Naimark dualities. Furthermore,
this example can not be  constructed by the existing spectral tetris construction \cite{CCHKP,
CFHWZ}, which is an algorithmic method of constructing sparse fusion frames utilized in the equal
rank characterization \cite{CFMUZ}. We end the paper by giving a complete list of maximal TFF
sequences for $\alpha \le 2$ in dimensions $N \le 9$.


\section{Basic majorization and duality results} \label{S2}

\begin{definition}
Fix a positive integer $N$. Let $L_1 \ge L_2 \ge \ldots \ge L_K>0$ be a weakly decreasing sequence
of positive integers. Such sequence is also known as a {\it partition} in number theory
\cite{Andrews76} and algebraic combinatorics \cite{Fulton97}. We say that $(L_1,L_2, \ldots, L_K)$
is a tight fusion frame (TFF) sequence if there exists orthogonal projections $P_1,\ldots, P_K$
such that
\begin{equation}\label{re0}
\alpha \mathbf I = \sum_{i=1}^K P_i, \qquad\text{and }\operatorname{rank}P_i =L_i,
\end{equation}
where $\alpha \in \RR$ and $\mathbf I$ is the identity on $\RR^N$. A trace argument shows that
$\alpha = \sum_{i=1}^K L_i/N \ge 1$. Given $\alpha \ge 1$ such that $\alpha N \in \NN$, we define
$\TFF(\alpha,N)$ to be the set of all TFF sequences in $\RR^N$ with the frame bound $\alpha$.
\end{definition}

\subsection{Majorization}
The following definition comes from the majorization theory of the Schur-Horn theorem, see
\cite{KW10}. In algebraic combinatorics the majorization partial order on partitions is known as
the {\it dominance order}, see \cite{Fulton97}.

\begin{definition}
Suppose that $\mathbf L =(L_1, L_2, \ldots ,L_K)$ and $\mathbf L'=(L'_1, L'_2 , \ldots ,L'_{K'})$
be two weakly decreasing sequences of non-negative integers. We say that $\mathbf L'$ majorizes
$\mathbf L$, and write $\mathbf L \preccurlyeq \mathbf L'$ if
\[
\sum_{i=1}^K L_i = \sum_{i=1}^{K'} L'_i \qquad\text{and}\qquad
\sum_{i=1}^k L_i \le \sum_{i=1}^k L'_i,
\]
for all $k \le \min(K,K')$.
\end{definition}

Observe that appending zeros at the tails of sequences $\mathbf L, \mathbf L'$ does not affect
majorization relation. Moreover, for sequences with only positive terms, the majorization $\mathbf
L \preccurlyeq \mathbf L'$ forces that $K \ge K'$.

The majorization principle for TFF sequences takes the following form.

\begin{theorem}\label{major}
Let $\mathbf L$ and $\mathbf L'$ be two weakly decreasing sequences of positive integers such that $\mathbf L \preccurlyeq \mathbf L'$.
Then, $\mathbf L'\in\TFF(\alpha,N)$ implies that $\mathbf L\in\TFF(\alpha,N).$
\end{theorem}

In the proof of Theorem \ref{major} we use the following elementary result on a sum of two projections.

\begin{lemma}\label{2proj}
Fix positive integers $p>q\ge 0$. Let $P$ and $Q$ be two orthogonal projection of ranks $p$ and
$q$, resp.  Then, there exists orthogonal projections $P'$ and $Q'$ of ranks $p-1$ and $q+1$,
resp., such that $P+Q=P'+Q'$.
\end{lemma}

\begin{proof}
Assume we have two projections $P$ and $Q$ with ranks $p>q$ that act on an $N$ dimensional vector
space $V$.  Then, we can decompose $V$ into the eigenspaces of $P$ and $Q$ such that
\[ V=V_P\oplus V_P^\perp, \qquad
V=V_Q\oplus V_Q^\perp,
\]
where $V_P$ and $V_P^\perp$ denote the $1$-eigenspace and $0$-eigenspace, resp.  Since $p>q$, we have that $p+(N-q)>N$ and hence
$\dim(V_P\cap V_Q^\perp)>0$.
Choose a nonzero vector in $V_P\cap V_Q^\perp$ and let $R$ denote the corresponding
rank $1$ projection.  Then, we can decompose
$P=\bar P + R$, where $\bar P$ is a rank $p-1$ projection.  Moreover, $Q+R$ is a
projection of rank $q+1$.
Thus, $P+Q=\bar P +(Q+R)$, which completes the proof of the lemma.
\end{proof}

\begin{proof}[Proof of Theorem \ref{major}]
Since $\mathbf L \preccurlyeq \mathbf L'$ we can find a sequence of partitions $\mathbf L = \mathbf L^0 \preccurlyeq \mathbf L^1 \preccurlyeq \ldots \preccurlyeq \mathbf L^n =\mathbf L'$ such that any two consecutive partitions $\mathbf L^{j-1}$ and $\mathbf L^j$, $j=1,\ldots,n$, differ at exactly two positions by $\pm 1$. That is, for each $j=1,\ldots,n$, there exist two positions $m< m'\in\NN$ such that
\begin{equation}\label{major2}
\begin{aligned}
\mathbf L^{j-1} &= (*,\ldots,*,\tilde L_m\phantom{+11},*,\ldots,*, \tilde L_{m'} \phantom{+11},*,\ldots,*),\\
\mathbf L^{j} &= (*,\ldots,*,\tilde L_{m}+1,*,\ldots,*,\tilde L_{m'}-1,*,\ldots,*),
\end{aligned}
\end{equation}
where the remaining values, denoted by $*$, are the same. Such $\mathbf L^j$'s can be easily constructed by the following recursive procedure.

Given the initial partitions $\mathbf L$ and $\mathbf L'$ we append extra zeros to $\mathbf L'$ so that $\mathbf L$ and $\mathbf L'$ have the same length. Define $m$ to be the first position such that initial subsequences $(L_1,\ldots, L_m)$ and $(L'_1,\ldots,L'_m)$ are not the same. Likewise, $m'$ is the last position such that the ending subsequences $(L_{m'},\ldots)$ and $(L'_{m'},\ldots)$ are not the same. Define $\mathbf L^1$ from $\mathbf L$ by replacing $L_m \to L_m+1$ and $L_{m'} \to L_{m'}-1$. It is not difficult to see that $\mathbf L^1$ forms a weakly decreasing sequence and $\mathbf L =\mathbf L_0 \preccurlyeq \mathbf L^1 \preccurlyeq \mathbf L'$.  Repeating this procedure recursively we define a sequence $\mathbf L^1 \preccurlyeq \mathbf L^2 \preccurlyeq \ldots \preccurlyeq \mathbf L'$. After a finite number of steps we must arrive at $\mathbf L^n=\mathbf L'$.

Observe that the ranks in \eqref{major2} satisfy $\tilde L_m \ge \tilde L_{m'}$. By Lemma \ref{2proj} applied to two projections with ranks $p=\tilde L_m +1> q=\tilde L_{m'}-1\ge 0$, if $\mathbf L^{j} \in \TFF(\alpha,N)$, then $\mathbf L^{j-1} \in \TFF(\alpha,N)$. Thus, a repetitive application of Lemma \ref{2proj} proves Theorem \ref{major}.
\end{proof}

We remark that the above proof does not use the tightness assumption in any way. Consequently, Theorem \ref{major} holds for general (not necessarily tight) fusion frames with a prescribed frame operator.

\subsection{Dualities}

In this subsection we shall establish two dualities for TFF sequences. The first duality involves taking orthogonal projections of the same ambient space and is a straightforward generalization of \cite[Theorem 6]{CFMUZ}.

\begin{theorem}\label{du1}
Suppose that $(L_1,L_2,\ldots, L_K) \in \TFF(\alpha,N)$. Then, $(N-L_K,N-L_{K-1},\ldots, N-L_1) \in \TFF(K-\alpha,N)$.
\end{theorem}

\begin{proof}
Let $P_1, \ldots, P_K$ be the orthogonal projections with $\operatorname{rank} P_i=L_i$ such that $\sum_{i=1}^K P_i = \alpha \mathbf I$. Clearly, $\sum_{i=1}^K (\mathbf I -P_i) = (K-\alpha) \mathbf I$ and $\operatorname{rank}(\mathbf I -P_i)=N-L_i$. This shows the theorem.
\end{proof}

The second result relies on taking more subtle orthogonal complements based on a dilation theorem for tight frames with bound $1$, also known as Parseval frames. It is known that every Parseval frame can be obtained as a projection of an orthogonal basis of some higher dimensional space. The complementary projection gives rise to another Parseval frame, which is often called the {\it Naimark's complement} of the original frame. This leads to the following result

\begin{theorem}\label{du2}
Suppose that $(L_1,L_2,\ldots, L_K)\in \TFF(\alpha,N)$. Then, the same sequence $(L_1,L_2,\ldots, L_K) \in \TFF(\tilde \alpha,\tilde N)$, where the dimension $\tilde N =(\sum_{i=1}^K L_i - N)$ and the frame bound $\tilde \alpha=\alpha/(\alpha-1)=\alpha N/\tilde N$.
\end{theorem}

\begin{proof} For each $k=0, \ldots, K$, define $\sigma_k =\sum_{i=1}^k L_i$ with the convention that $\sigma_0=0$. Our assumption implies that there exists a tight frame $\{v_j\}_{j=1}^{\sigma_K}$ in $\RR^N$ such that for each $k=1,\ldots, K$, the subcollection $\{v_j\}_{j=1+\sigma_{k-1}}^{\sigma_k}$ is an orthonormal sequence which spans the $L_k$ dimensional space $W_k$ from the definition of a TFF. Treating $v_1, \ldots, v_{\sigma_K}$ as column vectors we obtain an $N \times \sigma_K$ matrix $U$ with orthogonal rows each of norm $\alpha = \sigma_K/N$. This is due to the fact that $\{v_j\}_{j=1}^{\sigma_K}$ is a tight frame with constant $\alpha$.

Let $\tilde U$ be an extension of $U$ to a $\sigma_K \times \sigma_K$ matrix with all orthogonal rows of norm $\alpha$. In other words, $\frac{1}{\alpha} \tilde U$ is a unitary extension of $\frac{1}{\alpha} U$ which has orthonormal rows. Let $\{w_j\}_{j=1}^{\sigma_K}$ be the column vectors constituting the $(\sigma_K-N) \times \sigma_K$ submatrix of the bottom rows of $\tilde  U$. Since $\frac{1}{\alpha} \tilde U$ is  an orthogonal matrix we have
\[
\lan v_j, v_{j'} \ran + \lan w_j, w_{j'} \ran = \alpha \delta_{j,j'} \qquad \text{for all } j,j'=1,\ldots, \sigma_K.
\]
By the block orthogonality of $v_j$'s we have that for each block $k=1,\ldots, K$,
\[
\lan w_j, w_{j'} \ran = (\alpha-1) \delta_{j,j'} \qquad \text{for all } j,j'=1+\sigma_{k-1}, \ldots, \sigma_k.
\]
This means that the vectors $\{w_j\}_{j=1+\sigma_{k-1}}^{\sigma_k}$ form an orthogonal sequence which span some $L_k$ dimensional space $\tilde W_k$. Moreover, $\{w_j\}_{j=1}^{\sigma_K}$ is a tight frame with a constant $\alpha$ for $(\sigma_K - N)$ dimensional space. Consequently, unit norm vectors $\{\frac{1}{\alpha-1} w_j\}_{j=1}^{\sigma_K}$, which are block orthonormal, form a tight frame with a constant $\frac{\alpha}{\alpha-1}$. This leads to the decomposition $\tilde P_1 + \ldots \tilde P_K = \frac{\alpha}{\alpha-1}\mathbf I$, where $\tilde P_k$ is an orthogonal projection onto $\tilde W_k$. This completes the proof of the theorem.
\end{proof}

As an immediate corollary of Theorem \ref{du2} we can reduce the study of TFF sequences to the case when $1<\alpha< 2$; the case $\alpha =2$ does not cause any difficulties as we will see later.

\begin{corollary}\label{du3}
 If $\alpha >1$ is such that $\alpha N \in \NN$, then $\TFF(\alpha,N)=\TFF(\tilde \alpha, \tilde N)$, where $1/\alpha +1/\tilde \alpha = 1$ and $\tilde N = N(\alpha-1)$.
\end{corollary}

Observe that if there exists a TFF sequence with parameters $(\alpha,N)$, then by computing traces we necessarily have that $\alpha N \in \NN$. Hence, without loss of generality we shall always make this assumption.


\section{Estimates on first 3 ranks} \label{S4}

In this section we find necessary and sufficient conditions on the first three largest ranks of TFF projections. Our analysis is based on two fundamental results. Theorem \ref{fill} is due to Fillmore \cite[Theorem 1]{F69}. Lemma \ref{sum2} describes the spectral properties of the sum of two projections, and it can be thought of as a generalization of Lemma \ref{2proj}.

\begin{theorem}\label{fill}
A non-negative definite hermitian matrix $S$ is a sum of projections if and only if
\begin{equation}\label{fill1}
\operatorname{trace}(S)\in \mathbb N_0
\qquad\text{and}\qquad
\operatorname{trace}(S) \ge \operatorname{rank}(S).
\end{equation}
\end{theorem}

\begin{lemma}\label{sum2}
Let $P,Q$ be two orthogonal projections on an $N$ dimensional vector space $V$ with ranks $p,q$, resp.  For any $\lambda\in\RR,$ let $m(\lambda)$ be the multiplicity of $\lambda$ as an eigenvalue of $P+Q.$   Then, the following are true:
 \begin{enumerate}[(i)]
 \item $m(\lambda)>0 \implies \lambda\in [0,2]$,
 \item $\sum_{\lambda \in [0,2]} m(\lambda)=N$,
 \item $m(1)\geq |p-q|$,
 \item $\lambda\in (0,2) \implies m(\lambda)=m(2-\lambda)$,
 \item $m(0)-m(2)=N-p-q.$
 \end{enumerate}

Conversely, if $0\le p,q \le N$, and $m:\RR \to \NN_0$ satisfies (i)--(v), then there exists orthogonal projections $P,Q$ of ranks $p,q$, such that $m$ is a multiplicity function of $P+Q$.
 \end{lemma}

 \begin{proof}
Since $P,Q$ are hermitian, we can decompose $V$ as a direct sum of eigenspaces
$$V=V_P \oplus V^\perp_P=V_Q\oplus V^\perp_Q$$
where $V_P$ denotes the 1-eigenspace and $V_P^\perp$ the 0 eigenspace of $P$. Thus, $p=\dim(V_P)$ and $q=\dim(V_Q).$ Parts (i)--(iii) follow by basic linear algebra.

To prove part (iv) we define $f_{\lambda}:V\rightarrow V$ by
$$f_{\lambda}(v):=v_P+\left(\frac{\lambda}{\lambda-2}\right)v'_P,$$
where $v=v_P+v'_P$ is induced by the orthogonal decomposition $V=V_P \oplus V^\perp_P$ and $\lambda\in(0,2).$  Since $f_{\lambda}$ is an invertible and linear map, it suffices to show that if $(P+Q)v=\lambda v$, then $(P+Q)f_{\lambda}(v)=(2-\lambda)f_{\lambda}(v).$  Write
$$ v_P=x_Q+x'_Q\quad  \text{and}\quad  v'_P=y_Q+y'_Q$$ according to the decomposition $V=V_Q\oplus V^\perp_Q.$ Then,
$$(P+Q)v=v_P+x_Q+y_Q=2x_Q+y_Q+x'_Q=\lambda(x_Q+x'_Q+y_Q+y'_Q)$$ and hence
$$(2-\lambda)x_Q+(1-\lambda)y_Q=(\lambda-1)x'_Q +\lambda y'_Q.$$
This implies that
\begin{equation}\label{eq:xy_relation}
(2-\lambda)x_Q=(\lambda-1)y_Q\quad \text{and}\quad (1-\lambda)x'_Q=\lambda y'_Q\end{equation}
since $V_Q\cap V_Q^\perp=\{0\}.$

By equation \eqref{eq:xy_relation}, we have that
\begin{eqnarray*}
(P+Q)f_{\lambda}(v)&=&2x_Q+x'_Q +\left(\frac{\lambda}{\lambda-2}\right)y_Q\\
&=& (2-\lambda)v_P+\lambda x_Q+(\lambda-1)x'_Q+\left(\frac{\lambda}{\lambda-2}\right)y_Q\\
&=& (2-\lambda)v_P+\left(\frac{\lambda(1-\lambda)}{\lambda-2}\right)y_Q-\lambda y'_Q+\left(\frac{\lambda}{\lambda-2}\right)y_Q\\
&=& (2-\lambda)v_P-\lambda y_Q-\lambda y'_Q\\
&=& (2-\lambda)\bigg(v_P+\left(\frac{\lambda}{\lambda-2}\right) v'_P\bigg)=(2-\lambda)f_{\lambda}(v).
\end{eqnarray*}

This proves part (iv). To prove part (v), we consider the projection map $$g:V\rightarrow V_P+V_Q$$ where $V_P+V_Q$ denotes the span of vectors in $V_P, V_Q.$  We have that
$$\dim(V_P+V_Q)=\dim(V_P)+\dim(V_Q)-m(2)=p+q-m(2).$$
But $$\dim(V_P+V_Q)=N-\dim(\ker g)=N-m(0).$$  This shows that the properties (i)--(v) are necessary.

A quick way to see the converse direction is to utilize the honeycomb model of Knutson and Tao \cite{KT99, KT01}.
The honeycombs corresponding to triples $(P,Q,-(P+Q))$, where $p>q$ can be represented by one of the following diagrams. In the case $p=q$ the line corresponding the eigenvalue $-1$ of $-(P+Q)$ might not be present. We leave the details to the reader. This involves finding multiplicities of unlabelled line segments to satisfy the ``zero-tension'' property.

\begin{figure}[!ht]
$$ \begin{tikzpicture}[scale=1.2]
\draw (-3,1) node {$P$};\draw (3,1) node {$Q$};\draw (-2.8,-0.8) node {$p$};\draw (0.8,0.8) node {$q$};
\draw (-1.5,0.4) node {$N-p$};\draw (2.5,-0.1) node {$N-q$};
\draw (-2.3,-2.3) node {$m(2)$};\draw (2.3,-2.3) node {$m(0)$};
    \draw [->] (-1.732,-1) coordinate (a_1) -- (1.732,1) node [above,right] {$1$};
    \draw [->] (1.732,-1) coordinate (a_3) -- (-1.732,1) node [above,left] {$0$};;
    \draw [->] (0,-2) -- (3.464,0) node [above,right] {$0$};
    \draw [->] (a_3) -- (1.732,-3);
    \draw [->] (0,0) -- (0,-3);
    \draw [->] (0,-2) -- (-3.464,0) node [above,left] {$1$};
    \draw [->] (a_1) -- (-1.732,-3);
    \draw [->, dashed] (-0.866,-1.5) -- (-0.866,-3);
    \draw [->, dashed] (0.866,-1.5) -- (0.866,-3);
    \draw [dashed] (-0.866,-1.5) -- (0,-1);
    \draw [dashed] (0.866,-1.5) -- (0,-1);
    \draw [->, dashed] (-0.433,-1.75) -- (-0.433,-3);
    \draw [->, dashed] (0.433,-1.75) -- (0.433,-3);
    \draw [dashed] (-0.433,-1.75) -- (0,-1.5);
    \draw [dashed] (0.433,-1.75) -- (0,-1.5);
    \draw [->, dashed] (-1.299,-1.25) -- (-1.299,-3);
    \draw [->, dashed] (1.299,-1.25) -- (1.299,-3);
    \draw [dashed] (-1.299,-1.25) -- (0,-0.5);
    \draw [dashed] (1.299,-1.25) -- (0,-0.5);
\end{tikzpicture}$$
$$-2\ \, \lambda_1\ \lambda_2\, ...  -1\ ...\ \lambda_2'\ \lambda_1'\ 0\quad $$
$$-(P+Q)$$
\caption{Honeycomb with $m(2)> 0$, $m(0)> 0$ and $\lambda_i':=-2-\lambda_i.$}\end{figure}
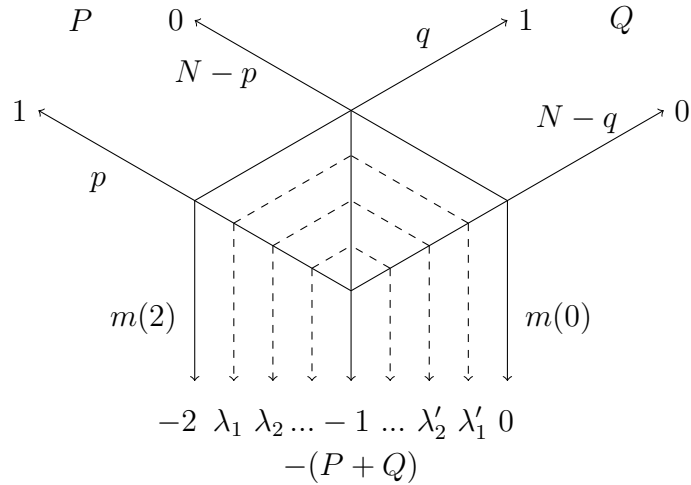

\begin{figure}[!t]
$$ \begin{tikzpicture}[scale=.8]
    \draw [->] (0,0) coordinate (a_1) -- (1.732,1);
    \draw [->] (1.732,-1) coordinate (a_3) -- (-1.732,1);
    \draw [->] (0,-2) -- (3.464,0);
    \draw [->] (a_3) -- (1.732,-3);
    \draw [->] (0,0) -- (0,-3);
    \draw [->] (0,-2) -- (-3.464,0);
    \draw [->, dashed] (-0.866,-1.5) -- (-0.866,-3);
    \draw [->, dashed] (0.866,-1.5) -- (0.866,-3);
    \draw [dashed] (-0.866,-1.5) -- (0,-1);
    \draw [dashed] (0.866,-1.5) -- (0,-1);
    \draw [->, dashed] (-0.433,-1.75) -- (-0.433,-3);
    \draw [->, dashed] (0.433,-1.75) -- (0.433,-3);
    \draw [dashed] (-0.433,-1.75) -- (0,-1.5);
    \draw [dashed] (0.433,-1.75) -- (0,-1.5);
    \draw [->, dashed] (-1.299,-1.25) -- (-1.299,-3);
    \draw [->, dashed] (1.299,-1.25) -- (1.299,-3);
    \draw [dashed] (-1.299,-1.25) -- (0,-0.5);
    \draw [dashed] (1.299,-1.25) -- (0,-0.5);
\end{tikzpicture}\quad \begin{tikzpicture}[scale=.8]
    \draw [->] (-1.732,-1) coordinate (a_1) -- (1.732,1);
    \draw [->] (0,0) -- (-1.732,1);
    \draw [->] (0,-2) -- (3.464,0);
    \draw [->] (0,0) -- (0,-3);
    \draw [->] (0,-2) -- (-3.464,0);
    \draw [->] (a_1) -- (-1.732,-3);
    \draw [->, dashed] (-0.866,-1.5) -- (-0.866,-3);
    \draw [->, dashed] (0.866,-1.5) -- (0.866,-3);
    \draw [dashed] (-0.866,-1.5) -- (0,-1);
    \draw [dashed] (0.866,-1.5) -- (0,-1);
    \draw [->, dashed] (-0.433,-1.75) -- (-0.433,-3);
    \draw [->, dashed] (0.433,-1.75) -- (0.433,-3);
    \draw [dashed] (-0.433,-1.75) -- (0,-1.5);
    \draw [dashed] (0.433,-1.75) -- (0,-1.5);
    \draw [->, dashed] (-1.299,-1.25) -- (-1.299,-3);
    \draw [->, dashed] (1.299,-1.25) -- (1.299,-3);
    \draw [dashed] (-1.299,-1.25) -- (0,-0.5);
    \draw [dashed] (1.299,-1.25) -- (0,-0.5);
\end{tikzpicture}$$
\caption{Honeycombs with $m(2)=0$ and $m(0)=0$, respectively.}\end{figure}
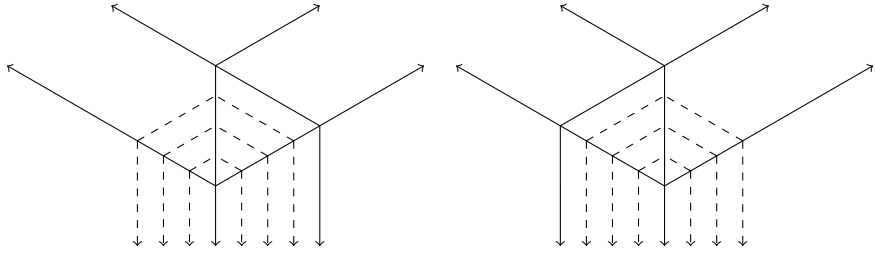
\end{proof}

Using Theorem \ref{fill} and Lemma \ref{sum2} our goal is to find necessary and sufficient conditions on the first three largest ranks of projections in a TFF.

\begin{theorem}\label{re}
Suppose that $1<\alpha<2$ and $(L_1\geq L_2\geq\cdots\geq L_K)\in \TFF(\alpha,N)$. Then, we have the following necessary conditions:
\begin{align}\label{re1}
L_1 & \le (\alpha-1)N,
\\
\label{re2}
L_1 + L_2 & \le N,
\\
\label{re3}
L_1 + L_2  + L_3 & \le \begin{cases}
N & \alpha<3/2, \\
2(\alpha-1)N & \alpha > 3/2.
\end{cases}
\end{align}

Conversely, if $L_1 \ge L_2 \ge L_3$ satisfy \eqref{re1}, \eqref{re2}, and \eqref{re3}, then there exists $\mathbf L \in \TFF(\alpha,N)$ which starts with the sequence $(L_1,L_2,L_3)$.
\end{theorem}

\begin{proof}
Suppose $\alpha \mathbf I$ is written as in \eqref{re0}. Then, $S=\alpha \mathbf I - P_1$ is an operator with 2 eigenvalues: $\alpha$ with multiplicity $N-L_1$ and $(\alpha-1)$ with multiplicity $L_1$. By Theorem \ref{fill} we must have that
\[
\alpha N - L_1 \ge N.
\]
Solving this for $L_1$ yields \eqref{re1}.

By Lemma \ref{sum2} the sum $P_1+P_2$ has eigenvalue $1$ with multiplicity at least $L_1-L_2$. Moreover, all other positive eigenvalues of this sum must come in pairs $(2-\lambda,\lambda)$, where $1 \le \lambda \le \alpha<2$. Thus, by Lemma \ref{sum2}(v), $L_1+L_2 \le N$. Let $S= \alpha \mathbf I - P_1 - P_2$.
By Theorem \ref{fill}, $S$ must satisfy \eqref{fill1}. Note that the trace of $S$ remains constant regardless of choices of $P_1$ and $P_2$,
\[
\operatorname{trace}(S)=\alpha N - L_1 - L_2.
\]
Thus, the rank of $S$ must be minimized to guarantee that it can be written as a sum of projections. The minimal rank of $S$ occurs if $P_1+P_2$ has eigenvalue $\alpha$ with multiplicity $L_2$, and thus eigenvalue $2-\alpha$ with the same multiplicity. Then, the rank of the corresponding $S$ is $N-L_2$. Thus, we have
\[
\alpha N - L_1 - L_2 \ge N-L_2.
\]
This leads again to \eqref{re1}. Thus, Fillmore's theorem does not introduce new constraints in this case. In other words, \eqref{re1} and \eqref{re2} are both necessary and sufficient conditions for the existence of an element of $\TFF(\alpha,N)$ starting with $(L_1,L_2)$.

Suppose next that $1<\alpha<3/2$. Repeating the above arguments, by Lemma \ref{sum2}, $P_1+P_2$ must have all of its $L_1+L_2$ non-zero eigenvalues (counted with multiplicities) in the interval $[2-\alpha,\alpha]$. Thus, if $L_1+L_2+L_3>N$, then at least one eigenvalue of $P_1+P_2+P_3$ would be at least $(2-\alpha)+1>3/2>\alpha$, which is impossible. Thus, \eqref{re3} is necessary.

To prove the converse, assume that $L_1+L_2+L_3 \le N$. Using honeycomb models as in the proof of  Lemma \ref{sum2} one can show that there exist projections $P_i$ such that their sum $P_1+P_2+P_3$ has the eigenvalue $\alpha$ with multiplicity $L_2+L_3$, and no eigenvalues bigger than $\alpha$.
This is shown in a two step process. First, we construct $P_2$ and $P_3$ such that their sum has eigenvalues: $\alpha$ and $2-\alpha$ both with multiplicities $L_3$ and $1$ with multiplicity $L_2-L_3$. Then, using a honeycomb model we can add on another projection $P_1$, such that $P_1+P_2+P_3$ has eigenvalue $\alpha$ with multiplicity $L_2+L_3$. This leads to an operator $S=\alpha \mathbf I - (P_1 + P_2 +P_3)$ with the  rank $N-L_2-L_3$.
The trace of $S$ remains constant regardless of the choice of such projections,
\[
\operatorname{trace}(S)=\alpha N - L_1 - L_2-L_3.
\]
Since $L_1 \le (\alpha-1)N$, Fillmore's Theorem \ref{fill} can be applied to represent $S$ as a sum of projections. This proves that \eqref{re1}--\eqref{re3} are both necessary and sufficient conditions for the first 3 ranks of a TFF sequence in the case $1<\alpha<3/2$. Unfortunately, the case $3/2<\alpha<2$ does not seem to be easily approachable with the techniques of this section. Instead, in Section \ref{S5} we shall give another combinatorial proof of Theorem \ref{re} which works in the entire range $1<\alpha<2$.
\end{proof}

We end this section by an explicit characterization of TFF sequences for some special values $\alpha$.

\begin{theorem}\label{rex} The set $\TFF(\alpha, N)$ has exactly one maximal element $\mathbf L$ with respect to majorization relation $\preccurlyeq$ in the following four cases indexed by $n \in \NN$:
\begin{align}
\label{rex1}
\alpha&=n,  \qquad
\mathbf L = (\underbrace{N,N,\ldots,N}_n),
\\
\label{rex2}
\alpha &= 1+\frac{1}{n}, \  n |N, \qquad \mathbf L = \bigg( \underbrace{\frac Nn,\frac Nn,\ldots, \frac Nn}_{n+1}\bigg),
\\
\label{rex3}
\alpha &= n+ \frac12, \ 2|N, \qquad \mathbf L = \bigg(\underbrace{N,\ldots,N}_{n-1},\frac N2, \frac N2, \frac N2\bigg),
\end{align}
\begin{equation}
\label{rex4}
\alpha = 1+\frac{2}{2n-1}, \ (2n-1)|N, \ \mathbf L = \bigg(\underbrace{\frac{2N}{2n-1},\ldots,\frac{2N}{2n-1}}_{n-1},\frac{N}{2n-1},\frac{N}{2n-1},\frac{N}{2n-1}\bigg).
\end{equation}
\end{theorem}

\begin{proof}
The case \eqref{rex1} is the easiest and it follows immediately from Theorem \ref{major}. The case \eqref{rex2} is obtained by the duality argument. Indeed, note that if $\alpha =1 +1/n$, then $n$ must divide $N$. Then, by Corollary \ref{du3}, $\TFF(\alpha, N)=\TFF(\tilde \alpha, \tilde N)$, where $\tilde \alpha =\alpha/(\alpha-1)=n+1$ and $\tilde N=(\alpha-1)N=N/n$.

In particular, we have that $\TFF(3/2,N)=\TFF(3,N/2)$ has a unique maximal element $(N/2,N/2,N/2)$. By appending $(n-1)$ $N$'s in the front of this sequence we obtain a maximal element of $\TFF(n+1/2,N)$. It remains to show that this is the only maximal element.

Suppose that we have another element $(L_1, \ldots, L_K)\in \TFF(n+1/2,N) $. Let $P_i$'s be the corresponding projections. Given two hermitian matrices $S$ and $T$ we write $S \le T$ if $\lan Sx,x \ran \le \lan Tx,x\ran$ for all $x\in\RR^N$.
Since $\sum_{i=1}^n P_i \le n \mathbf I$, $S=\sum_{i=n+1}^K P_i$ must have full rank $N$. By Fillmore's Theorem \ref{fill}, this implies that
\[
\operatorname{trace}(S) = \sum_{i=n+1}^K L_i \ge N.
\]
Thus, $L_1+ \ldots +L_n \le (n-1/2)N$.

Suppose on the contrary that $L_1+\ldots + L_{n+1} > nN$. Let $W_i$'s be the corresponding subspaces with $\dim W_i=L_i$. By basic linear algebra the intersection satisfies
\[
\dim \bigg(\bigcap_{i=1}^{n+1} W_i \bigg) = L_1+\ldots + L_{n+1} - nN> 0.
\]
This implies that $P_1 + \ldots + P_{n+1}$ has eigenvalue $n+1$ exceeding $\alpha=n+1/2$, which is a contradiction. Thus, we have necessarily that $L_1+\ldots + L_{n+1} \le nN$. Clearly,
\[
L_1+\ldots + L_{n+2} \le L_1+ \ldots + L_K =(n+1/2)N.
\]
Consequently, $(L_1, \ldots, L_K) \preccurlyeq \mathbf L$ proving \eqref{rex3}.

Finally, \eqref{rex4} is shown by the duality argument. Indeed, note that if $\alpha =1 +2/(2n-1)$, then $2n-1$ must divide $N$. Then, by Corollary \ref{du3}, $\TFF(\alpha, N)=\TFF(\tilde \alpha, \tilde N)$, where $\tilde \alpha =\alpha/(\alpha-1)=n+1/2$ and $\tilde N=(\alpha-1)N=2N/(2n-1)$.
\end{proof}

Section \ref{S7} provides the list of all maximal elements in $\TFF(\alpha,N)$ for all $\alpha\le 2$ and dimensions $N\le 9$. It is easy to observe that all unique maximal elements in our tables are covered by Theorem \ref{rex}. Hence, it is very tempting to conjecture that for general $\alpha$ and $N$, if $\TFF(\alpha,N)$ has only one maximal element, then $\alpha$ must necessarily come from  one of the four cases of Theorem \ref{rex}.


\section{A combinatorial characterization of tight fusion frames} \label{S3}

In this section we give a combinatorial characterization of tight fusion frames in the context of
Schur functions.  The main result of this section, Theorem \ref{th:combchar}, is a direct
consequence of Horn's recursion for the hermitian eigenvalue problem (for a survey of this problem
see \cite{Fu00}).  For completeness, we state the main results of this body of work.  For any partition $$\lambda=(\lambda_1\geq\lambda_2\geq\cdots\geq \lambda_d>0),$$ let
$$|\lambda|=\sum_{i=1}^d \lambda_i$$ denote the size of $\lambda$ and let $d$ denote the length.  We say $\lambda$ is a rectangular partition if $\lambda=(a^b):=\underbrace{(a,\ldots,a)}_b$ for some positive integers $a,b.$  For any partition $\lambda$, let $s_{\lambda}$ denote the corresponding Schur polynomial.  The polynomial $s_{\lambda}$ is a homogeneous polynomial of degree $|\lambda|.$  It is well known that the Schur polynomials form a linear basis of the ring of symmetric polynomials with integer coefficients.  Hence for any collection of partitions $\lambda^1,\ldots, \lambda^K$ we can define the corresponding \emph{Littlewood-Richardson coefficients}
$c(\lambda^1,\ldots,\lambda^K; \mu)$ as the product structure constants of

$$\prod_{i=1}^K s_{\lambda^i}=\sum_{\mu}  c(\lambda^1,\ldots,\lambda^K; \mu)\, s_{\mu}.$$

The Littlewood-Richardson coefficients defined above play an important role in the hermitian
eigenvalue problem.  To state these results, we first need some notation. There is a standard
identification between sets of positive integers of size $r$ and partitions of length at most $r$.
For any set $I=\{i_1<i_2<\cdots<i_r\},$ define the partition
$$\lambda(I):=(i_r-r,i_{r-1}-r+1,\ldots,i_1-1).$$

Let $(\beta^1,\ldots,\beta^{K+1})\in (\RR^{N})^{K+1}$ denote a collection of sequences where each
$\beta^i:=(\beta^i_1\geq\cdots \geq\beta^i_N)$.  The goal of the hermitian eigenvalue problem is to
determine for which sequences $(\beta^1,\ldots,\beta^{K+1})$ do there exist $N\times N$ hermitian
matrices $H_1,\ldots,H_{K+1}$ such that the eigenvalues of $H_i$ are given by the sequence
$\beta^i$ and
$$\sum_{i=1}^K H_i=H_{K+1}.$$

The following theorem, proved by Klyachko in \cite{Kly98}, gives a remarkable characterization in
terms of collection of a inequalities parametrized by non-zero Littlewood-Richardson coefficients.

\begin{theorem}\label{th:klyachko}Let $(\beta^1,\ldots,\beta^{K+1})\in (\RR^{N})^{K+1}$ be a collection of sequences of non-increasing real numbers such that
$$\sum_{i=1}^K\sum_{j=1}^N\beta^i_j=\sum_{j'=1}^N\beta^{K+1}_{j'}.$$
Then the following are equivalent:
\begin{enumerate}
\item There exist $N\times N$ hermitian matrices $H_1,\ldots,H_{K+1}$ with spectra $(\beta^1,\ldots,\beta^{K+1})$ such that $$\sum_{i=1}^K H_i=H_{K+1}.$$

\item For every $r<N$, the sequence $(\beta^1,\ldots,\beta^{K+1})$ satisfies the inequality
\begin{equation}\label{eq:Horn}\sum_{i=1}^K\sum_{j\in I^j}\beta^i_j\geq \sum_{j'\in I^{K+1}}\beta^{K+1}_{j'}\end{equation}
for every collection of subsets $I^1,\ldots,I^{K+1}$ of size $r$ of the integers $\{1,2,\ldots, N\}$ where the Littlewood-Richardson coefficient $$c(\lambda(I^1),\ldots,\lambda(I^K);\lambda(I^{K+1}))\neq 0.$$
\end{enumerate}\end{theorem}

The inequalities given in \eqref{eq:Horn} are called Horn's inequalities and were initially defined
in a very different way by Horn in \cite{Ho62}.  While Horn's list of inequalities in \cite{Ho62}
are, a priori, different than Klyachko's list \eqref{eq:Horn}, they were shown to be equivalent as a
consequence of the saturation theorem of Knutson and Tao in \cite{KT99}.  What is amazing about
this equivalence is that Horn's initial definition of the inequalities \eqref{eq:Horn} uses a
recursion unrelated to Littlewood-Richardson coefficients.  Horn's recursion in light of Theorem
\ref{th:klyachko} can be stated as follows:

\begin{theorem}\label{th:hornrec}Let $I^1,\ldots,I^{K+1}$ be subsets of size $r$ of the integers $\{1,2,\ldots, N\}$ such that
\begin{equation}\label{eq:necess}\sum_{i=1}^K\sum_{j=1}^r\lambda(I^i)_j=\sum_{j'=1}^r\lambda(I^{K+1})_{j'}.\end{equation} The following are equivalent:
\begin{enumerate}
\item The Littlewood-Richardson coefficient $$c(\lambda(I^1),\ldots,\lambda(I^K);\lambda(I^{K+1}))\neq 0.$$

\item There exist $r\times r$ hermitian matrices  $H_1,\ldots,H_{K+1}$ with spectra $(\lambda(I^1),\ldots, \linebreak
\lambda(I^{K+1}))$ such that \begin{equation}\label{eq:matrixsum}\sum_{i=1}^K H_i=H_{K+1}.\end{equation} \end{enumerate}
\end{theorem}

The recursion says that a collection of subsets $I^1,\ldots,I^{K+1}$ corresponds to a Horn
inequality if and only if the corresponding collection of  partitions are eigenvalues of some $r\times r$
hermitian matrices which satisfy \eqref{eq:matrixsum}.  Hence Horn's inequalities can be defined recursively by induction on $N$.  We
also remark that equation \eqref{eq:necess} is a necessary condition for the corresponding
Littlewood-Richardson coefficient to be nonzero.

\smallskip

We now apply Theorem \ref{th:hornrec} to the case of tight fusion frames.  Suppose that $(L_1\geq
L_2\geq\cdots \geq L_K)\in \TFF(\alpha,N)$ and that $M:=\sum_{i=1}^K L_i.$  Then there exist
orthogonal projections $P_1,\ldots, P_K$ such that

\begin{equation}\label{eq:1} \sum_{i=1}^K NP_i=M\mathbf I.\end{equation}

Since $P_i$ is an orthogonal projection, the spectra of the hermitian matrix $NP_i$ is given by
$$(\underbrace{N,\ldots,N}_{L_i},\underbrace{0,\ldots,0}_{N-L_i}).$$
Let $(N^{L_i})$ denote the corresponding rectangular partition to the spectra above.  The following
is a direct corollary of Theorem \ref{th:hornrec}.

\begin{theorem}\label{th:combchar}Fix an integer $N$ and let $(L_1\geq L_2\cdots \geq L_K)$ be a sequence of nonnegative integers such that $L_1\leq N.$  Let $M:=\sum_{i=1}^K L_i$ and $\alpha=M/N.$  The following are equivalent:
\begin{enumerate}
\item The sequence $(L_1\geq L_2\geq\cdots \geq L_K)\in \TFF(\alpha, N)$.

\item The Littlewood-Richardson coefficient $$c((N^{L_1}),\ldots,(N^{L_K});(M^N))\neq 0.$$
\end{enumerate}\end{theorem}

\begin{proof}Assume part (1).  Then there exist orthogonal projections $P_1,\ldots, P_K$ with ranks $(L_1,\ldots, L_K)$ such that
\begin{equation}\label{eq:2}\sum_{i=1}^K P_i=\alpha\mathbf I.\end{equation}
Multiplying both sides of equation \eqref{eq:2} by $N$ gives equation \eqref{eq:1}.  Applying Theorem \ref{th:hornrec} gives part (2).

\smallskip

Conversely, if we assume part (2) then by Theorem \ref{th:hornrec}, there exists a collection of
$N\times N$ matrices which satisfy equation \eqref{eq:1} and have spectra
$(N^{L_1}),\ldots,(N^{L_K})$. Scaling by $1/N$ yields the desired tight fusion frame.
\end{proof}

\bigskip

The condition that $c((N^{L_1}),\ldots,(N^{L_K});(M^N))\neq 0$ can be made computationally explicit
by the following existence condition.  With the notation of Theorem \ref{th:combchar} we consider
the following properties for an $N\times M$ matrix $A=A[i,j]$.

\smallskip

\begin{itemize}
\item[(i)] (integral nonnegativity) $A[i,j]\in \ZZ_{\geq 0}$

\item[(ii)] (row sum) $\displaystyle\sum_{j=1}^M A[i,j]=M\quad \forall i$

\item[(iii)] (column sum) $\displaystyle\sum_{i=1}^N A[i,j]=N\quad \forall j$

\item[(iv)] (row sum dominance) $\displaystyle\sum_{j=1}^l (A[i,j]-A[i+1,j])\geq A[i+1,l+1]\quad \forall i,l$

\item[(v)] (column sum dominance) $\displaystyle \sum_{i=1}^l (A[i,j]-A[i,j+1])\geq A[l+1,j+1]\quad \forall j,l$
\end{itemize}

\smallskip

Observe that properties $(iv)$ and $(v)$ require dominance with one additional summand in the later
row or column. Also note that $(ii)$ and $(iii)$ are the only properties dependant on the size of
the matrix $A$.  Let $A$ be an $N\times M$ matrix and consider the sequence $(L_1,\ldots L_K)$. We
can partition $A$ into a sequence of column block matrices $$A=[A_1|A_2|\cdots| A_K]$$ where each
$A_i$ is the corresponding $N\times L_i$ sub-matrix of $A.$  We now have the following addition to
Theorem \ref{th:combchar}.

\begin{corollary}\label{th:combcharcor}Conditions (1) and (2) in Theorem \ref{th:combchar} are equivalent to the
following:
\begin{enumerate}
\item[(3)] There exists an $N\times M$ matrix $A$ which satisfies properties (i)-(iv) and whose column
block sub-matrices $A_1,\ldots, A_K$ each satisfy property (v). \end{enumerate}

Moreover, the coefficient $c((N^{L_1}),\ldots,(N^{L_K});(M^N))$ equals the number of $N\times M$ matrices $A$ which satisfy (3).
\end{corollary}

\begin{proof}We refer to \cite{Fulton97} for definitions and details of Littlewood-Richardson skew tableaux.  Consider the Littlewood-Richardson coefficients $c_{\lambda,\mu}^{\nu}$ corresponding to the product of two Schur functions

$$s_{\lambda}s_{\mu}=\sum_{\nu}c_{\lambda,\, \mu}^{\nu}\ s_{\nu}.$$

It is well known that the number $c_{\lambda,\, \nu}^{\mu}$ is precisely equal to the number of
Littlewood-Richardson skew tableaux $\nu/ \lambda$ of content $\mu.$  Now suppose there exists a
$N\times M$ matrix $A$ which satisfies the conditions of Corollary \ref{th:combcharcor} with
respect to a sequence $\mathbf L=(L_1,\cdots, L_K).$  For any $k\leq K$ let
$$A(k):=[A_1|\cdots|A_k]$$ denote the submatrix of $A$ consisting of the matrices $A_1, \ldots,
A_k.$  By properties $(i)$ and $(iv),$ the row sums of $A(k)$ yield a partition
\begin{equation}\label{eq:mupartition}
\mu^k:=\left(\sum_j A(k)[i,j]\right)_{i=1}^{N}\end{equation} given in the standard weakly
decreasing form. It is easy to see that $\mu^k/ \mu^{k-1}$ is a well defined skew partition.
Consider the Young diagram corresponding to $\mu^k/ \mu^{k-1}.$  We can fill the boxes of the
$j^{th}$ row of this diagram with $A_k[j,1]$ 1's, $A_k[j,2]$ 2's, $A_k[j,3]$ 3's and so forth in
weakly increasing order.  Properties $(iv)$ and $(v)$ imply that the resulting skew tableau is a
Littlewood-Richardson skew tableau.  Property $(iii)$ implies that content of the tableau is that
of the rectangular partition $(N^{L_k}).$  Hence the existence of the matrix $A(k)$ implies that
the Littlewood-Richardson coefficient
$$c_{\mu^{k-1},\, (N^{L_k})}^{\mu^k}\neq 0.$$
Finally, properties $(ii)$ and $(iii)$ imply that $\mu^K=(M^N).$  By
induction on $k,$ multiplying the Schur functions $s_{(N^{L_1})},\ldots s_{(N^{L_K})}$ gives that
$$c((N^{L_1}),\ldots,(N^{L_K});(M^N))\neq 0.$$  It is easy to see that this argument can be
reversed.  This bijection together with Littlewood-Richardson rule for counting $c_{\lambda,\,
\nu}^{\mu}$ implies that second part of Corollary \ref{th:combcharcor}.  This completes the
proof.\end{proof}

\begin{example}\label{ex1}
We consider two examples where tight fusion frames exist for $N=5$ and $M=8.$

\smallskip

First, consider the sequence $\mathbf L=(2,2,2,2).$  The following matrix
$$A=\left( \begin{array}{cc|cc|cc|cc}
5&0&3&0&0&0&0&0 \\
0&5&0&1&2&0&0&0 \\
0&0&2&2&2&2&0&0 \\
0&0&0&2&1&0&5&0 \\
0&0&0&0&0&3&0&5 \\
\end{array} \right)$$
satisfies the conditions in Corollary \ref{th:combcharcor}.  We write out the corresponding Young tableaux to the partitions
$\mu^1, \mu^2, \mu^3$ and $\mu^4$ with content given by the sub-matrices $A(1)$, $A(2)$, $A(3)$,
$A(4)$:

\bigskip

$$\begin{young}
?1& ?1 & ?1 & ?1 & ? 1 &,&,&,\\
?2& ?2 & ?2 & ?2 & ? 2\\
,\\ ,\\ ,\\
\end{young}\qquad \begin{young}
? 1& ?1 & ?1 & ?1 & ?1 & !! 1& !! 1&!!1\\
? 2& ? 2& ?2 & ?2 & ? 2& !! 2\\
!! 1& !! 1& !! 2& !!2 \\
!! 2& !! 2\\ ,\\
\end{young}$$

$$\begin{young}
? 1& ?1 & ?1 & ?1 & ?1 & !! 1& !! 1&!!1\\
? 2& ? 2& ?2 & ?2 & ? 2& !! 2& ?? 1&??1\\
!! 1& !! 1& !! 2& !!2 & ??1 & ??1 & ?? 2&??2\\
!! 2& !! 2& ??1\\
?? 2& ??2 & ??2\\
\end{young}\qquad \begin{young}
? 1& ?1 & ?1 & ?1 & ?1 & !! 1& !! 1&!!1\\
? 2& ? 2& ?2 & ?2 & ? 2& !! 2& ?? 1&??1\\
!! 1& !! 1& !! 2& !!2 & ??1 & ??1 & ?? 2&??2\\
!! 2& !! 2& ??1& ! 1& ! 1& ! 1& ! 1&!1\\
?? 2& ??2 & ??2 & ! 2& !2 & !2 & !2 &!2\\
\end{young}$$

\bigskip

Note that the all the data can be encoded in the final partition $\mu^4$ as a union of skew Littlewood-Richardson tableaux.

\smallskip

For the second example, we consider $\mathbf L=(3,2,1,1,1)$ and the matrix

$$A=\left( \begin{array}{ccc|cc|c|c|c}
5&0&0&3&0&0&0&0 \\
0&5&0&0&3&0&0&0 \\
0&0&5&0&0&3&0&0 \\
0&0&0&2&0&2&4&0 \\
0&0&0&0&2&0&1&5 \\
\end{array} \right)$$

\bigskip

The corresponding union of Littlewood-Richardson tableaux is given by

\bigskip

$$\begin{young}
?1& ?1& ?1& ?1& ?1& !! 1& !! 1&!!1\\
?2& ?2& ?2& ?2& ?2& !! 2& !! 2&!!2\\
?3& ?3& ?3& ?3& ?3& ??1 & ?? 1&??1\\
!! 1& !!1& ??1& ?? 1& ! 1& ! 1& ! 1&!1\\
!! 2& !!2& !1 & ??? 1& ???1 & ???1 & ???1 &???1\\
\end{young}$$

\end{example}


\section{Combinatorial majorization and hook type sequences} \label{S5}

In this section we give alternate proofs of Theorem \ref{major} on majorization and Theorem \ref{re} on estimates using the combinatorics of Schur functions and Theorem \ref{th:combchar}.  We begin with some fundamental definitions and lemmas on Schur functions.  Let $\lambda\subseteq (M^N).$  We define the dual partition of $\lambda$ in  $(M^N)$ to be the partition

$$\lambda^*:=(M-\lambda_N\geq M-\lambda_{N-1}\geq\cdots\geq M-\lambda_1).$$

\smallskip

$$\begin{young}
\ynoright\ynobottom&\ynobottom&\ynobottom&\ynobottom&&\yframe\ynoright\ynobottom&\ynobottom\\
\ynobottom\ynotop&\ynoframe\lambda&\ynoframe&\ynotop&\yframe\ynoright\ynobottom&\ynotop\ynobottom&\ynotop\ynobottom\\
\ynotop&\ynotop&\ynotop &\yframe\ynoright\ynobottom&\ynoframe&\lambda^*\ynoframe&\ynotop\ynobottom\\
&&\ynotop&\ynotop&\ynotop&\ynotop&\ynotop\\
\end{young}$$

\smallskip

\begin{lemma}\label{lemma:hooks}Let $\lambda\subseteq (M^N)$ and let $p(\lambda)$ denote the number of parts of $\lambda$ equal to $M$.   Assume that for some positive integer $k$ we have that
 $$|\lambda|=N(M-k).$$  Then
 $$c(\lambda,\underbrace{(N),\ldots,(N)}_k;(M^N))\neq 0$$ if and only if $k\geq N-p(\lambda).$\end{lemma}

\begin{proof}
 The lemma follows from two elementary facts about Schur functions.  Consider the product
$$(s_{(N)})^k=\sum_{\mu} c((N),\ldots,(N);\mu)\ s_{\mu}$$
By the Pieri rule, we have that $c((N),\ldots,(N);\mu)\neq 0$ if and only if $\mu$ has length less than or equal to $k$ and $|\mu|=Nk.$  Furthermore, if $\lambda,\mu\subseteq (M^N)$, then $c^{(M^N)}_{\lambda,\, \mu}\neq 0$ if and only if $\mu=\lambda^*.$  It is easy to check that $\lambda^*$ appears as a summand in the product $(s_{(N)})^k$ precisely when $k\geq N-p(\lambda).$\end{proof}


The following theorem on the product of Schur functions corresponding to rectangular partitions is proved by Okada in \cite[Theorem 2.4]{Okada98}.

\begin{theorem}\label{th:Okada}Fix integers $a,b,N_1,N_2$ with $a\geq b.$ The product of Schur functions

\begin{equation}\label{eq:rectangle}
s_{(N_1^a)}s_{(N_2^b)}=\sum_{\lambda} s_{\lambda},
\end{equation}
where the sum is over all partitions $\lambda$ with length $ \le a+b$ such that
\begin{itemize}
\item $\lambda_{b+1}=\lambda_{b+2}=\cdots=\lambda_{a}=N_1.$
\item $\lambda_b\geq\max\{N_1,N_2\}.$
\item $\lambda_i+\lambda_{a+b+1-i}=N_1+N_2\qquad \forall i\in\{1,\ldots, b\}$\end{itemize}\end{theorem}

We now give an alternate proof of Theorem \ref{major} using Theorem \ref{th:Okada} in the case when
$N_1=N_2.$

\begin{lemma}Fix a positive integer $N$ and let $0<a<b.$  Then the Littlewood-Richardson coefficients $$c_{(N^b),\, (N^a)}^{\lambda}\leq c_{(N^{b-1}),\, (N^{a+1})}^{\lambda}.$$ In particular, Theorem \ref{major} on majorization of tight fusion frames follows.\end{lemma}

\begin{proof}It is easy to check the $\lambda$ that appear in the summation \eqref{eq:rectangle} for the pair $((N^b), (N^a))$ are contained in the $\lambda$ that appear in the summation \eqref{eq:rectangle} for the pair $((N^{b-1}), (N^{a+1})).$  This proves the inequality.  The application to tight fusion frames follows from Theorem \ref{th:combchar}.\end{proof}

It is easy to see that by majorization, the following theorem is equivalent to Theorem \ref{re} on
estimates.

\begin{theorem}\label{th:hooks}Assume the conditions in Theorem \ref{th:combchar}.  Further assume that $\alpha=M/N< 2.$  If $(L_1\geq L_2\geq\cdots\geq L_K)\in \TFF(\alpha,N),$ then we have the following necessary conditions.

\begin{enumerate}
\item $L_1\leq M-N.$

\smallskip

\item $L_1+L_2\leq N$

\smallskip

\item If $\alpha>3/2$, then $L_1+L_2+L_3\leq 2(M-N).$

\smallskip

\item If $\alpha<3/2$, then $L_1+L_2+L_3\leq N$

\end{enumerate}

\smallskip

Conversely, suppose $L_1,L_2,L_3$ satisfy the above conditions and $L_4=\cdots=L_K=1.$  Then
$(L_1\geq L_2\geq\cdots\geq L_K)\in \TFF(\alpha,N).$\end{theorem}

\begin{proof}Recall that for any partition $\lambda\subseteq (M^N)$, we let $p(\lambda)$ denote the number of parts of $\lambda$ equal to $M.$  First we prove part (1).  By majorization, it suffices to assume that $L_2=1.$  Part (1) now follows from Lemma \ref{lemma:hooks} by setting $\lambda =(N^{L_1})$ and observing that $p((N^{L_1}))=0.$

\smallskip

We now prove part (2).  By majorization, it suffices to assume that $L_3=1.$  Consider the product

\begin{equation}\label{eq:schurL1L2}s_{(N^{L_1})}s_{(N^{L_2})}=\sum_{\lambda} s_{\lambda}.\end{equation}

By Theorem \ref{th:Okada}, we have that $\lambda_1+\lambda_{L_1+L_2}=2N$ for every $\lambda$ in the
sum \eqref{eq:schurL1L2}.  If $\lambda\subseteq (M^N)$, then $\lambda_1\leq M.$ Hence
$$\lambda_{L_1+L_2}=2N-\lambda_1\geq 2N-M>0$$ since $\alpha<2.$  This implies that $L_1+L_2\leq N$
since $(M^N)$ has only $N$ parts.

\smallskip

For part (3), we assume that $L_4=1.$  First, if  $L_2+L_3\leq L_1,$  then by part (1),
$L_1+L_2+L_3\leq 2(M-N).$  Next, we assume $L_1\leq L_2+L_3.$  Consider the product

\begin{equation}\label{eq:schurL1L2L3}s_{(N^{L_1})}s_{(N^{L_2})}s_{(N^{L_3})}=\sum_{\lambda} c((N^{L_1}),(N^{L_2}),(N^{L_3});\lambda)\  s_{\lambda}\end{equation}

Since $\alpha>3/2,$ for any $\lambda\subseteq (M^N)$ such that
$c((N^{L_1}),(N^{L_2}),(N^{L_3});\lambda)\neq 0,$ we have that $p(\lambda)\leq L_1.$ This can be seen by considering $L_2$ and $L_3$ as large as possible, hence $L_1=L_2=L_3$. One can show using the Littlewood-Richardson rule that since $3N<2M$, 3 layered bricks of width $N$ cannot span $M$ more than once, see diagram below.

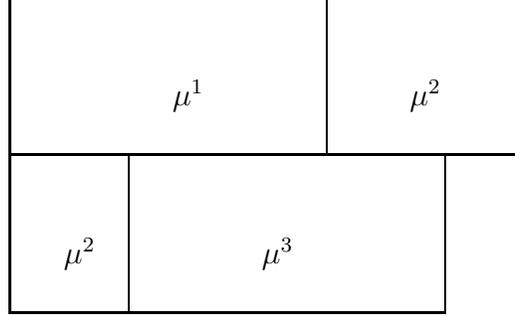
\centerline{
\begin{young}
]=  \ynobottom & \ynobottom & \ynobottom & \ynobottom & \ynobottom & \ynobottom & \ynobottom & \ynobottom & ]= \ynobottom & \ynobottom & \ynobottom & \ynobottom & \ynobottom
\\
\ynobottom\ynotop & , &  , & , & , & , & , & , & \yframe\ynobottom\ynotop\ynoright & , & , & , & \ynobottom\ynotop
\\
\ynobottom\ynotop & , & , & , & $\mu^1$ \ynoframe & , & , & , & \yframe\ynotop\ynoright\ynobottom & , & $\mu^2$ \ynoframe & , &\ynobottom\ynotop
\\
\ynotop & \ynotop & \ynotop & \ynotop & \ynotop & \ynotop & \ynotop & \ynotop & \yframe\ynotop\ynoright & \ynotop & \ynotop & \ynotop & \ynotop
\\
\ynobottom\ynotop & , & ,  & \yframe\ynobottom\ynotop\ynoright & , & , & , & , & , & , & ,  & \yframe\ynobottom\ynoright\ynotop
\\
]= \ynobottom\ynotop & , & ,  & \yframe\ynobottom\ynotop\ynoright & , & , & , & , & , & , & ,  & \yframe\ynobottom\ynoright\ynotop
\\
]= \ynobottom\ynotop & \ynotop\ynobottom $\mu^2$ & ,  & \yframe\ynobottom\ynotop\ynoright & , & , &  \ynotop\ynobottom $\mu^3$ & , & , & , & ,  & \yframe\ynobottom\ynoright\ynotop
\\
]= \ynotop & \ynotop & \ynotop  & \yframe\ynotop\ynoright & \ynotop &\ynotop &\ynotop &\ynotop &\ynotop &\ynotop &\ynotop
\end{young}}

By Lemma
\ref{lemma:hooks},
$$M-L_1-L_2-L_3\geq N-p(\lambda)\geq N-L_1.$$
Hence $L_2+L_3\leq M-N.$  This
proves part (3).

\smallskip

For part (4), fix any $\lambda$ in the summand found in equation $\eqref{eq:schurL1L2}$ such that $\lambda\subseteq (M^N).$  Since $\alpha< \frac{3}{2}$, we have that
$$\lambda_{L_1+L_2}=2N-\lambda_1\geq 2N-M> M-N.$$
Hence the rectangular partition $((M-N+1)^{L_1+L_2})\subseteq \lambda.$  Comparing the two products
\begin{equation}\label{eq:schur3}s_{\lambda}s_{(N^{L_3})}= \sum_{\mu'} c_{\lambda, (N^{L_3})}^{\mu'}\ s_{\mu'}\end{equation}
and
\begin{equation}\label{eq:schur4}s_{((M-N+1)^{L_1+L_2})}s_{(N^{L_3})}=\sum_{\mu} s_{\mu}\end{equation}
we have that any partition $\mu'$ from equation \eqref{eq:schur3} such that $c_{\lambda, (N^{L_3})}^{\mu'}\neq 0$ contains some $\mu$ from equation \eqref{eq:schur4}.  Therefore it is enough to consider the partitions $\mu$ from \eqref{eq:schur4}.  By Theorem \ref{th:Okada}, we get that
$$\mu_1+\mu_{L_1+L_2+L_3}=M-N+1+N=M+1$$ for every $\mu$ in the sum \eqref{eq:schur4}.  Hence if $\mu\subseteq (M^N)$, then $\mu_{L_1+L_2+L_3}>0$ since $\mu_1\leq M.$  Thus $L_1+L_2+L_3\leq N.$  This proves part (4).

\smallskip

To prove sufficiency, we construct $\lambda$ in the sum \eqref{eq:schurL1L2L3} such that $\lambda\subseteq (M^N)$ and $c((N^{L_1}),(N^{L_2}),(N^{L_3});\lambda)\neq 0.$ One can show  using the Littlewood-Richardson rule that parts (1)--(4) imply that such a $\lambda$ exists. Furthermore, we can construct $\lambda$ such that $p(\lambda)=L_1$ if $L_1\leq L_2+L_3$ or $p(\lambda)=L_2+L_3$ if $L_2+L_3\leq L_1$, see Figures 3 and 4.
In either case, Lemma \ref{lemma:hooks} implies that we only need to check that $L_1+L_2+L_3\leq 2(M-N).$ However, this is already a necessary condition.  This completes the proof of the theorem.\end{proof}

\begin{figure}[!ht]
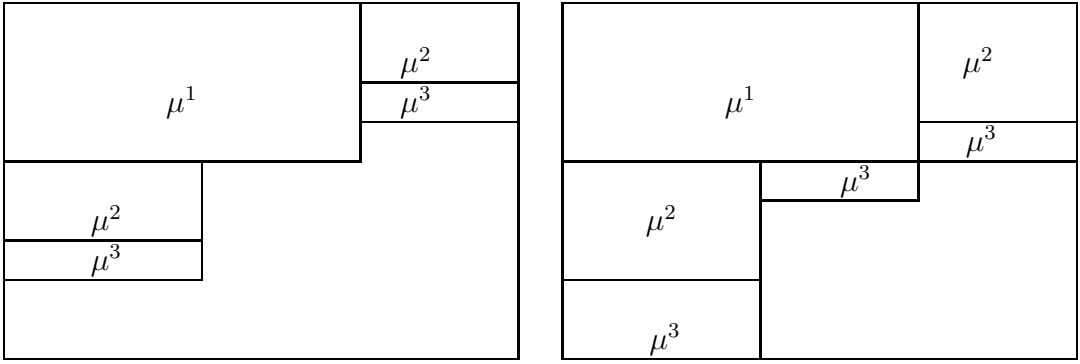

\label{fig3}
\centerline{
\begin{young}
]=  \ynobottom & \ynobottom & \ynobottom & \ynobottom & \ynobottom & \ynobottom & \ynobottom& \ynobottom & \ynobottom & ]= \ynobottom &  \ynobottom  &  \ynobottom & \ynobottom
\\
\ynobottom\ynotop & , &  , & , & , & , & , & , & , & \yframe\ynotop\ynoright  & $\mu^2$ \ynotop & \ynotop &  \ynotop
\\
\ynobottom\ynotop & , & , & , & $\mu^1$ \ynoframe & , & , & , & , & \yframe\ynotop\ynoright &  $\mu^3$ \ynotop &\ynotop & \ynotop
\\
\ynotop & \ynotop & \ynotop & \ynotop & \ynotop & \ynotop & \ynotop & \ynotop & \ynotop & \yframe\ynotop\ynoright\ynobottom & , & , & \ynotop\ynobottom
\\
\ynobottom\ynotop & , & , & , & ,  & \yframe\ynotop\ynobottom\ynoright & , & , & , & , & , & , & \ynobottom\ynotop
\\
\ynotop & \ynotop & \ynotop $\mu^2$ & \ynotop & \ynotop  & \yframe\ynotop\ynoright\ynobottom & , & , & , & , & ,  & , & \ynotop\ynobottom \\
\ynotop & \ynotop & \ynotop $\mu^3$ & \ynotop & \ynotop  & \yframe\ynotop\ynoright\ynobottom & , & , & , & , & ,  & , & \ynotop\ynobottom \\
\ynotop\ynobottom & , & ,  & , & , & , & , & , & , & , & , &  , & \ynobottom\ynotop
\\
\ynotop &\ynotop &\ynotop  &\ynotop &\ynotop &\ynotop &\ynotop &\ynotop &\ynotop  &\ynotop &\ynotop &\ynotop &\ynotop
\end{young}
\quad
\begin{young}
]=  \ynobottom & \ynobottom & \ynobottom & \ynobottom & \ynobottom & \ynobottom & \ynobottom& \ynobottom & \ynobottom & ]= \ynobottom &  \ynobottom  &  \ynobottom & \ynobottom
\\
\ynobottom\ynotop & , &  , & , & , & , & , & , & , & \yframe\ynobottom\ynotop\ynoright  & $\mu^2$ \ynoframe & , &  \ynobottom\ynotop
\\
\ynobottom\ynotop & , & , & , & $\mu^1$ \ynoframe & , & , & , & , & \yframe\ynotop\ynoright &  \ynotop &\ynotop & \ynotop
\\
\ynotop & \ynotop & \ynotop & \ynotop & \ynotop & \ynotop & \ynotop & \ynotop & \ynotop & \yframe\ynotop\ynoright & \ynotop $\mu^3$ & \ynotop & \ynotop
\\
\ynobottom\ynotop & , & , & , & ,  & \yframe\ynotop\ynoright & \ynotop & $\mu^3$ \ynotop & \ynotop & \yframe\ynotop\ynobottom\ynoright & , & , & \ynobottom\ynotop
\\
\ynotop\ynobottom & , & ,  $\mu^2$ & , & , & \yframe\ynotop\ynoright\ynobottom & , & , & , & , & , &  , & \ynobottom\ynotop
\\
\ynotop & \ynotop & \ynotop & \ynotop & \ynotop  & \yframe\ynotop\ynoright\ynobottom & , & , & , & , & ,  & , & \ynotop\ynobottom \\
\ynotop\ynobottom & , & ,  & , & , & \yframe\ynotop\ynoright\ynobottom & , & , & , & , & , &  , & \ynobottom\ynotop
\\
\ynotop &\ynotop &\ynotop $\mu^3$ &\ynotop &\ynotop &\yframe\ynotop\ynoright &\ynotop &\ynotop &\ynotop  &\ynotop &\ynotop &\ynotop &\ynotop
\end{young}
}
\caption{Construction of $\lambda$  for $\alpha<3/2$ as a union of Littlewood-Richardson skew tableaux $\mu^1,\mu^2,\mu^3$ when $ L_2+L_3 \le L_1$ and $L_1 \le L_2 + L_3$, resp. This construction is possible since $L_1+L_2 + L_3 \le N$.} \end{figure}

\begin{figure}[!ht]
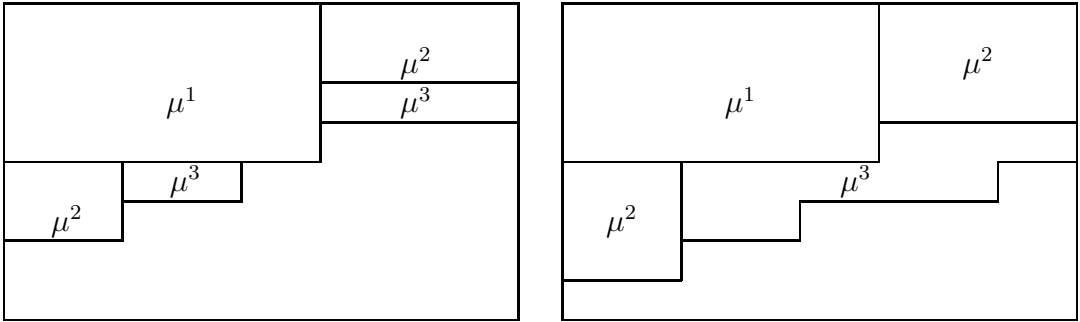

\label{fig4}
\centerline{
\begin{young}
]=  \ynobottom & \ynobottom & \ynobottom & \ynobottom & \ynobottom & \ynobottom & \ynobottom & \ynobottom & ]= \ynobottom & \ynobottom & \ynobottom & \ynobottom & \ynobottom
\\
\ynobottom\ynotop & , &  , & , & , & , & , & , & \yframe\ynotop\ynoright & \ynotop & $\mu^2$ \ynotop & \ynotop & \ynotop
\\
\ynobottom\ynotop & , & , & , & $\mu^1$ \ynoframe & , & , & , & \yframe\ynotop\ynoright & \ynotop & $\mu^3$ \ynotop  & \ynotop &\ynotop
\\
\ynotop & \ynotop & \ynotop & \ynotop & \ynotop & \ynotop & \ynotop & \ynotop & \yframe\ynotop\ynobottom\ynoright & , & ,  & , & \ynobottom\ynotop
\\
\ynobottom\ynotop & , & ,  & \yframe\ynotop\ynoright & \ynotop $\mu^3$ & \ynotop & \yframe\ynotop\ynobottom\ynoright & , &, & , & ,  & , & \ynobottom\ynotop
\\
\ynotop & \ynotop $\mu^2$ & \ynotop & \yframe\ynotop\ynobottom\ynoright & , & , & , & , & , & , & , & , & \ynobottom\ynotop\\
\ynotop\ynobottom & , & ,  & , & , & , & , & , & , & , & , & , & \ynotop\ynobottom
\\
\ynotop &\ynotop &\ynotop &\ynotop &\ynotop &\ynotop &\ynotop &\ynotop &\ynotop &\ynotop &\ynotop &\ynotop &\ynotop
\end{young}
\quad
\begin{young}
]=  \ynobottom & \ynobottom & \ynobottom & \ynobottom & \ynobottom & \ynobottom & \ynobottom & \ynobottom & ]= \ynobottom & \ynobottom & \ynobottom & \ynobottom & \ynobottom
\\
\ynobottom\ynotop & , &  , & , & , & , & , & , & \yframe\ynobottom\ynotop\ynoright & , & $\mu^2$ \ynoframe & , & \ynobottom\ynotop
\\
\ynobottom\ynotop & , & , & , & $\mu^1$ \ynoframe & , & , & , & \yframe\ynotop\ynoright & \ynotop & \ynotop & \ynotop &\ynotop
\\
\ynotop & \ynotop & \ynotop & \ynotop & \ynotop & \ynotop & \ynotop & \ynotop & \yframe\ynobottom\ynotop\ynoright & , & , & , & \ynobottom\ynotop
\\
\ynobottom\ynotop & \ynobottom\ynotop & \ynobottom\ynotop  & \yframe\ynobottom\ynotop\ynoright & \ynobottom\ynotop & \ynobottom\ynotop & \ynotop & $\mu^3$ \ynotop & \ynotop & \ynotop & \ynotop  & \yframe\ynobottom\ynoright & \ynobottom
\\
\ynotop\ynobottom & , $\mu^2$ & , & \yframe\ynotop\ynoright & \ynotop & \ynotop & \yframe\ynotop\ynoright\ynobottom & , & , & , & , & , & \ynobottom\ynotop\\
\ynotop & \ynotop & \ynotop  & \yframe\ynotop\ynoright\ynobottom & , & , & , & , & , & , & , & , & \ynotop\ynobottom \\
\ynotop &\ynotop &\ynotop &\ynotop &\ynotop &\ynotop &\ynotop &\ynotop &\ynotop &\ynotop &\ynotop &\ynotop &\ynotop
\end{young}
}
\caption{Construction of $\lambda$  for $\alpha>3/2$ as a union of Littlewood-Richardson skew tableaux $\mu^1,\mu^2,\mu^3$ when $L_2+L_3 \le L_1$ and $L_1 \le L_2 + L_3$, resp.}\end{figure}

\begin{remark}Parts (2) and (4) of Theorem \ref{th:hooks} can be generalized to the following statement.

\smallskip

Let $2\leq k\leq K.$  If $\displaystyle\alpha<\frac{k}{k-1}$, then $L_1+\cdots +L_{k}\leq N.$

\smallskip

The proof follows the same argument as the proof of Theorem \ref{th:hooks} part (4).\end{remark}


\section{Combinatorial spatial and Naimark duality} \label{S6}

Theorems \ref{du1} and \ref{du2} establish spatial and Naimark dualities for tight fusion frames.
By Theorem \ref{th:combchar}, we have the analogous results for Littlewood-Richardson coefficients.

\begin{corollary}Assume we have a sequence of integers $(L_1\geq \cdots\geq L_K)$ as in Theorem \ref{th:combchar}.  Then
$$c((N^{L_1}),\ldots,(N^{L_K});(M^N))\neq 0 \Leftrightarrow c((N^{N-L_1}),\ldots,(N^{N-L_K});((KN-M)^{N}))\neq 0$$
and
$$c((N^{L_1}),\ldots,(N^{L_K});(M^N))\neq 0 \Leftrightarrow c(((M-N)^{L_1}),\ldots,((M-N)^{L_K});(M^{(M-N)}))\neq 0.$$
\end{corollary}

In this section we prove a much stronger version of the corollary above.  In particular, we prove
that these Littlewood-Richardson coefficients are equal.  We will frequently reference properties
$(i)-(v)$ for matrices defined in the paragraph preceding Corollary \ref{th:combcharcor} using
lower case roman numerals.  We first consider spatial duality.

\begin{theorem}\label{th:combSp}The Littlewood-Richardson coefficients
\begin{equation}\label{eq:Spacomb}c((N^{L_1}),\ldots,(N^{L_K});(M^N))=c((N^{N-L_1}),\ldots,(N^{N-L_K});((KN-M)^{N})).\end{equation}
\end{theorem}

The coefficient $c((N^{L_1}),\ldots,(N^{L_K});(M^N))$ is precisely the number of $N\times M$
matrices $A$ which satisfy the conditions given in the Corollary \ref{th:combcharcor}.  We will
call such a collection of matrices the set of configuration matrices corresponding to $(L_1,\ldots,
L_K; N)$.  We prove Theorem \ref{th:combSp} by providing a bijection between the configuration
matrices corresponding to the coefficients in \eqref{eq:Spacomb}.

\smallskip

Suppose that $c((N^{L_1}),\ldots,(N^{L_K});(M^N))\neq 0$ and fix a configuration matrix
$A=[A_1|A_2| \cdots|A_K].$  For each $A_i$, we construct a $N\times (N-L_i)$ matrix $B_i$ as
follows.  Decompose $$A_i=\sum_{j=1}^N C_j$$ as a sum of binary matrices which satisfy the
following conditions for all integers $y,j$

\begin{enumerate}
 \item $\displaystyle \sum_{x=1}^N C_j[x,y]=1$
 \item $\displaystyle \sum_{x=1}^{N'} (C_j[x,y]- C_j[x,y+1])\geq 0\quad \forall\ N'<N$
 \item $\displaystyle \sum_{x=1}^{N'} (C_j[x,y]- C_{j+1}[x,y])\geq 0\quad \forall\ N'<N.$
\end{enumerate}

Consider $A_2$ from Example \ref{ex1}.  We have that

$$\left( \begin{array}{cc}3&0 \\0&1 \\2&2 \\0&2 \\0&0 \\\end{array} \right)=
 \left( \begin{array}{cc}1&0 \\0&1 \\0&0 \\0&0 \\0&0 \\\end{array} \right)+
 \left( \begin{array}{cc}1&0 \\0&0 \\0&1 \\0&0 \\0&0 \\\end{array} \right)+
 \left( \begin{array}{cc}1&0 \\0&0 \\0&1 \\0&0 \\0&0 \\\end{array} \right)+
 \left( \begin{array}{cc}0&0 \\0&0 \\1&0 \\0&1 \\0&0 \\\end{array} \right)+
 \left( \begin{array}{cc}0&0 \\0&0 \\1&0 \\0&1 \\0&0 \\\end{array} \right).$$

It is easy to see that this decomposition of $A_i$ is unique since $A_i$ satisfies properties $(i),
(iii)$ and $(v).$  For each $C_j$, define the $N\times (N- L_i)$ matrix $C'_j$ to be the unique
binary matrix which satisfies conditions $(1), (2)$ and that $[C_j|C'_j]$ is invertible.  For
example, if $N=5$ then

$$C_j=\left( \begin{array}{cc}1&0 \\0&0 \\0&1 \\0&0 \\0&0 \\\end{array} \right)\quad \rightsquigarrow\quad
 C'_j=\left( \begin{array}{ccc}0&0&0 \\1&0&0 \\0&0&0 \\0&1&0 \\0&0&1 \\\end{array} \right).$$
Define $$B_i:=\sum_{i=1}^N C'_j$$ and consider the $N\times(KN-M)$ matrix
$$B:=[B_K|B_{K-1}|\cdots|B_1].$$

Note that the binary decomposition of $B_i$ into $C'_j$ also satisfies conditions $(1)-(3)$ if we
order the $C'_j$ in reverse.   Moreover, if we apply this algorithm to the matrix $B,$ we will
recover the matrix $A$.  We now record some important observations on the submatrices $A_i$ and
$B_i.$ First, if $x<y$, then
\begin{equation}\label{eq:spa1}A_i[x,y]=B_i[x,y]=0.\end{equation}  Second, we have that
\begin{equation}\label{eq:spa2}A_i[x,y]+B_i[x,x-y]=A_i[x+1,y+1]+B_i[x+1,x-y+1].\end{equation}  In the equations above we take $A_i[x,y]=0$ (resp. $B_i[x,y]=0$) if $x,y$ lie outside the boundaries of $A_i$ (resp. $B_i$).  In the case when $x=y,$ we get
\begin{equation}\label{eq:spa3}A_i[x,x]=A_i[x+1,x+1]+B_i[x+1,1].\end{equation}
Theorem \ref{th:combSp} follows from the proceeding proposition.

\begin{proposition}The matrix $N\times(KN-M)$ matrix $B$ is a configuration matrix for the sequence $(N-L_K,\ldots, N-L_1; KN-M).$ \end{proposition}

\begin{proof}The most challenging part of this proof is to show that the matrix $B$ satisfies property $(iv).$
Hence the majority of this argument is dedicated to the proof this property.  We first consider the
other properties.  Properties $(i)-(iii)$ are immediate by construction of $B$.  Property $(v)$
follows form the fact that each $B_i$ is a sum of binary matrices which satisfy conditions
$(1)-(3).$  We now prove that $B$ satisfies property $(iv)$ by contradiction.  Suppose there exists
integers $i,l$ such that
\begin{equation}\label{eq:spa4}\sum_{j=1}^l (B[i,j]-B[i+1,j])<B[i+1,l+1].\end{equation}
We define the integers $k,l'$ as follows.  Let $k$ denote largest integer for which the partial sum
$$l':=\sum_{j=1}^k (N-L_{K-j+1})\leq l.$$ Hence $l'$ is the number of
columns of the submatrix $[B_K|\cdots |B_{K-k+1}]$ of $B$.

\smallskip

Observe that each row sum of the matrix $[A_j|B_j]$ is equal to $N.$ Combining this observation
with equation \eqref{eq:spa4} gives that
\begin{eqnarray*}\sum_{j=1}^l (B[i,j]-B[i+1,j])&=&\sum_{j=1}^{l'} (B[i,j]-B[i+1,j])+\sum_{j=l'+1}^l (B[i,j]-B[i+1,j])\\
&=&\sum_{j=M-(kN-l'-1)}^{M}(A[i+1,j]-A[i,j])\\
&&\hspace{.5in}+\sum_{j=l'+1}^l (B[i,j]-B[i+1,j])<B[i+1,l+1].\end{eqnarray*} Rewriting this
inequality yields
\begin{eqnarray*}\sum_{j=M-kN+l'+1}^{M}(A[i+1,j]-A[i,j])&<& B[i+1,l+1]-\sum_{j=l'+1}^l (B[i,j]-B[i+1,j])\\
&=&B[i+1,l'+1]+\sum_{j=l'+1}^l (B[i+1,j+1]-B[i,j]) .\end{eqnarray*} The matrix entries of $B$
appearing on the right hand side of the above equation are all contained in the submatrix
$B_{K-k}.$  Applying equations \eqref{eq:spa2},\eqref{eq:spa3}, we get that
\begin{equation}\label{eq:spapf2}\sum_{j=M-kN+l'+1}^{M}(A[i+1,j]-A[i,j])<\sum_{j=0}^{l-l'}
(A_{K-k}[i,i-j]-A_{K-k}[i+1,i-j+1]).\end{equation} By equation \eqref{eq:spa1}, $A_{K-k}[x,y]=0$ if
$y>x.$  Hence we can extend the right hand side of equation \eqref{eq:spapf2} to
\begin{multline*}
\sum_{j=M-kN+l'+1}^{M}(A[i+1,j]-A[i,j])<A_{K-k}[i,i-l+l']\\
+\sum_{j=0}^{L_{K-k}-(i+1)+(l-l')} (A_{K-k}[i,L_{K-k}-j]-A_{K-k}[i+1,L_{K-k}-j]).
\end{multline*}
Now the fact that $A$ satisfies properties $(ii),$ contradicts the fact that it also satisfies
property $(iv).$  This completes the proof.
\end{proof}

\begin{example}\label{ex:sp} Let $N=4$ and consider the sequence $\mathbf L=(2,2,2,1).$  By Corollary \ref{th:combcharcor}, the matrix $A$ below implies that $\mathbf L\in\TFF(7/4,4).$
$$A=\left( \begin{array}{cc|cc|cc|c}
4&0&3&0&0&0&0 \\
0&4&0&1&2&0&0 \\
0&0&1&2&2&2&0 \\
0&0&0&1&0&2&4 \\
\end{array} \right)$$
We get that
$$B=\left( \begin{array}{ccc|cc|cc|cc}
4&0&0&4&0&1&0&0&0 \\
0&4&0&0&2&2&1&0&0 \\
0&0&4&0&0&1&0&4&0\\
0&0&0&0&2&0&3&0&4 \\
\end{array} \right)$$
and hence $(3,2,2,2)\in\TFF(9/4,4).$
\end{example}


We now give the analogous theorem on combinatorial Naimark duality.

\begin{theorem}\label{th:combNai}The Littlewood-Richardson coefficients
\begin{equation}\label{eq:Naicomb}c((N^{L_1}),\ldots,(N^{L_K});(M^N))=c(((M-N)^{L_1}),\ldots,((M-N)^{L_K});(M^{(M-N)})).\end{equation}
\end{theorem}

As with Theorem \ref{th:combSp}, we define a bijection between configuration matrices corresponding
to the Littlewood-Richardson coefficients in \eqref{eq:Naicomb}.  Fix a configuration matrix $A$
corresponding to the sequence $(L_1,\ldots,L_K;N)$ and consider the Littlewood-Richardson skew
tableaux $\mu^k/ \mu^{k-1}$ where $\mu^k$ is defined in equation \eqref{eq:mupartition}.  To each
$\mu^k/\mu^{k-1}$ we define the $L_k\times M$ binary matrix $T_k$ by

$$T_k[x,y]:=\begin{cases}1 & \text{if $x$ appears in column $y$ of $\mu^k/\mu^{k-1}$}\\
0 &\text{otherwise}.\end{cases}$$

The partition shape of $\mu^k$ can be recovered from the matrices $T_1,\ldots T_K$ as follows.
Define the matrix $T(k)$ by ``stacking" the matrices $T_1,.\ldots, T_k$ (see Example \ref{ex:Nai}
below).  In other words,
$$T(k):= \left( \begin{array}{c} T_1\\ \hline \vdots\\ \hline T_k\end{array} \right).$$  Since $A$
satisfies property $(iv)$, the partition $\mu^k$ can be recovered by upward justifying the nonzero
entries of $T(k).$  In particular, the entire collection $T_1,\ldots, T_K$ uniquely determines the
matrix $A.$

\smallskip

We now define the ``complementary" $M\times L_k$ matrix $S_k$ by $$S_k[x,y]:=1-T_k[x,M-y+1]$$ and
$S(k)$ as the corresponding column matrix with block entries $S_1,\ldots, S_k.$  It is easy that if
the nonzero entries of $S(k)$ are justified upwards, we get the dual partition $(\mu^k)^*$ in
rectangle $(M^{M_k})$ where $M_k:=\sum_{i=1}^k L_k.$   Hence $S_1,\ldots, S_K$ determines some
matrix $B$ in the same way that $T_1,\ldots, T_K$ determines $A$.  Also note that we can recover
$T_k$ from $S_k$ by applying the complementary operation to $S_k.$  Theorem \ref{th:combNai}
follows from the proceeding proposition.

\begin{proposition}The collection $S_1,\ldots, S_K$ determines a configuration matrix for the sequence $(L_1,\ldots, L_K; M-N).$\end{proposition}

\begin{proof}Let $B=[B_1|\cdots| B_K]$ denote the matrix corresponding to the collection
$S_1$, $\ldots$, $S_K$.  We will show that $B$ is a configuration matrix for the sequence
$(L_1,\ldots, L_K; M-N).$  In this case, property $(v)$ is the most challenging to prove.  Hence
most the argument to dedicated to this part of the proof.

\smallskip

First, note that $B$ trivially satisfies properties $(i)$ and $(iv).$  Next, we observe that $A$
satisfies properties $(ii)$ and $(iii)$ if and only if the matrix $T(K)$ has $M$ columns where each
column sum is equal to $N.$  Since $S(K)$ has the same number of columns as $T(K)$ with column sums
of $M-N$, we get that $B$ also satisfies properties $(ii)$ and $(iii).$

\smallskip

We now prove that $B$ satisfies property $(v)$ by contradiction.  Suppose there exists $B_k$ and
integers $j,l$ such that
$$\sum_{i=1}^l (B_k[i,j]-B_k[i,j+1])< B_k[l+1,j+1].$$
This implies there exists an integer $l'$ such that
\begin{equation}\label{eq:combnai1}\sum_{i=l'+1}^{M}(S_k[j,i]-S_k[j+1,i])<0\end{equation}
with \begin{equation}\label{eq:combnai2}S_k[j,l'+1]=0\quad\text{and}\quad
S_k[j+1,l'+1]=1.\end{equation} Conversely, assume there exists an integer $l'$ such that equations
\eqref{eq:combnai1} and \eqref{eq:combnai2} are true.  By equation \eqref{eq:combnai2}, there
exists an integer $l''$ such that
$$\sum_{i=l'+1}^{M}S_k[j,i]=\sum_{i=1}^{l''} B_k[i,j]\quad\text{and}\quad\sum_{i=l'+1}^{M}S_k[j+1,i]\leq \sum_{i=1}^{l''+1} B_k[i,j+1].$$
Hence by equation \eqref{eq:combnai1},
$$-B_k[l''+1,j+1]+\sum_{i=1}^{l''} (B_k[i,j]-B_k[i,j+1])\leq \sum_{i=l'+1}^{M}(S_k[j,i]-S_k[j+1,i])<0.$$  Observe that if \eqref{eq:combnai1}
is true for $l$, then there is always some integer $l'\leq l$ for which both \eqref{eq:combnai1}
and \eqref{eq:combnai2} are true.   Thus the failure of property $(v)$ is equivalent to equation
\eqref{eq:combnai1}.  By definition of $S_k$ and equation \eqref{eq:combnai1}, we have that
$$\sum_{i=1}^{M-l'}(T_k[j+1,i]-T_k[j,i])<0.$$  Since the row sums of $T_k$ equal $N$,
$$\sum_{i=M-l'+1}^{M}(T_k[j,i]-T_k[j+1,i])<0.$$  Therefore the matrix $A$ also fails to satisfy property $(v)$ which is a contradiction.
This completes the proof.\end{proof}

\begin{example}\label{ex:Nai} Consider $N=4$ and $\mathbf L=(2,2,2,1)$ as in Example \ref{ex:sp}.
Then $\mu^4,$ as a union of Littlewood-Richardson skew tableaux, is equal to

$$\begin{young}
? 1& ?1 & ?1 & ?1 & !! 1& !! 1&!!1\\
? 2& ? 2& ?2 & ?2&  !! 2& ?? 1&??1\\
!! 1& !! 2& !! 2 & ??1 & ??1 & ?? 2&??2\\
!! 2& ??2& ?? 2& ! 1& ! 1& ! 1&!1\\
\end{young}$$

We have that

$$ T(4)= \left( \begin{array}{ccccccc}
1&1&1&1&0&0&0\\
1&1&1&1&0&0&0\\ \hline
1&0&0&0&1&1&1\\
1&1&1&0&1&0&0\\ \hline
0&0&0&1&1&1&1\\
0&1&1&0&0&1&1\\ \hline
0&0&0&1&1&1&1\\
\end{array} \right)\quad \rightsquigarrow\quad
S(4)= \left( \begin{array}{ccccccc}
1&1&1&0&0&0&0\\
1&1&1&0&0&0&0\\ \hline
0&0&0&1&1&1&0\\
1&1&0&1&0&0&0\\ \hline
0&0&0&0&1&1&1\\
0&0&1&1&0&0&1\\ \hline
0&0&0&0&1&1&1\\
\end{array} \right)$$

Upward justifying the nonzero entries of $S(4)$ gives the union of Littlewood-Richardson skew
tableaux

$$\begin{young}
? 1& ?1 & ?1 & !! 1& !! 1&!!1& ??1\\
? 2& ? 2& ?2 & !!2&  ?? 1&??1& ??2\\
!!2& !!2& ?? 2 & ??2 & !1 &!1&!1\\
\end{young}$$

The corresponding configuration matrix is
$$B=\left( \begin{array}{cc|cc|cc|c}
3&0&3&0&1&0&0 \\
0&3&0&1&2&1&0 \\
0&0&0&2&0&2&3 \\
\end{array} \right).$$
and hence $(2,2,2,1)\in \TFF(7/3,3).$\end{example}


\section{Examples and tables of TFF sequences} \label{S7}

This section is divided into two parts.  In the first part we give several examples of existence of
tight fusion frames using skew Littlewood-Richardson tableaux as in Example \ref{ex1}.  In the
second part, we give a complete list of tight fusion frame sequences for $N\leq 9$ and $\alpha\leq
2$ by listing all maximal elements in the partial order induced by majorization.

\subsection{Examples of skew Littlewood-Richardson tableaux}  The following are some examples of Littlewood-Richardson tableaux in the cases of $N=3, 5,
7.$  Readers who are interested in combinatorial spatial and Naimark duality as discussed in Section
\ref{S6} are encouraged to apply the bijective constructions to these examples.

\bigskip
$N=3$ and $\mathbf L=(3,2,1),$ $\mathbf L=(2,1,1,1),$ and $\mathbf L=(1,1,1,1),$

$$\begin{young}
? 1& ? 1& ? 1& !! 1& !! 1& !!1 \\
? 2& ? 2& ? 2& !! 2& !! 2& !!2 \\
? 3& ? 3& ? 3& ?? 1& ?? 1&?? 1\\
\end{young}\hspace{.5in}
\begin{young}
? 1& ?1 & ?1 & !! 1& !! 1 \\
? 2& ?2 & ?2 & ??1 &??1  \\
!! 1& ?? 1& ! 1& !1 & ! 1\\
\end{young}
\hspace{.5in}
\begin{young}
? 1& ? 1& ? 1& !! 1 \\
!!1& !!1 & ?? 1& ?? 1\\
?? 1& ! 1& ! 1& ! 1\\
\end{young}$$

\bigskip

$N=5$ and $\mathbf L=(2,2,2,2)$ and $\mathbf L=(3,3,3,3)$

$$\begin{young}
? 1& ?1 & ?1 & ?1 & ?1 & !! 1& !! 1&!!1\\
? 2& ? 2& ?2 & ?2 & ? 2& !! 2& ?? 1&??1\\
!! 1& !! 1& !! 2& !!2 & ??1 & ??1 & ?? 2&??2\\
!! 2& !! 2& ??1& ! 1& ! 1& ! 1& ! 1&!1\\
?? 2& ??2 & ??2 & ! 2& !2 & !2 & !2 &!2\\
\end{young}\hspace{.5in}
\begin{young}
?1&?1&?1&?1&?1&!!1&!!1&!!1&!!1&!!1&??1&??1\\
?2&?2&?2&?2&?2&!!2&!!2&!!2&??1&??2&??2&??2\\
?3&?3&?3&?3&?3&!!3&??1&!1&!1&!1&!1&!1\\
!!2&!!2&!!3&!!3&??1&??2&??2&!2&!2&!2&!2&!2\\
!!3&!!3&??3&??3&??3&??3&??3&!3&!3&!3&!3&!3\\
\end{young}$$

\bigskip

$N=7$ and $\mathbf L=(4,3,3,1,1)$ and $\mathbf L=(3,2,2,2,1)$

$$\begin{young}
?1&?1&?1&?1&?1&?1&?1&!!1&!!1&!!1&!!1&!!1\\
?2&?2&?2&?2&?2&?2&?2&!!2&!!2&!!2&!!2&!!2\\
?3&?3&?3&?3&?3&?3&?3&!!3&!!3&??1&??1&??1\\
?4&?4&?4&?4&?4&?4&?4&??1&??1&??2&??2&??2\\
!!1&!!1&!!3&!!3&!!3&??1&??1&??2&??2&??3&??3&??3\\
!!2&!!2&??2&??2&??3&??3&!1&!1&!1&!1&!1&!1\\
!!3&!!3&??3&??3&!1&???1&???1&???1&???1&???1&???1&???1\\
\end{young}\hspace{.5in}
\begin{young}
?1&?1&?1&?1&?1&?1&?1&!!1&!!1&!!1\\
?2&?2&?2&?2&?2&?2&?2&!!2&!!2&!!2\\
?3&?3&?3&?3&?3&?3&?3&??1&??1&??1\\
!!1&!!1&!!1&!!1&??1&??1&??1&??2&??2&??2\\
!!2&!!2&!!2&!!2&??2&!1&!1&!1&!1&!1\\
??1&??2&??2&!1&!1&!2&!2&!2&!2&!2\\
??2&!2&!2&???1&???1&???1&???1&???1&???1&???1\\
\end{young}$$

\subsection{Tables of maximal tight fusion frames.}

At the end of this section we give a complete list of tight fusion frames for $N\leq 9$ and $\alpha
\leq 2$ by listing all maximal elements in the partial order induced by majorization.  These lists are
generated by applying the techniques developed in this paper.  In particular, we use the following methods

\begin{itemize}
\item Constructing Littlewood-Richardson tableaux as in Corollary \ref{th:combcharcor}.
\item Recursive construction using spatial and Naimark duality.
\item Recursive construction using Lemma \ref{lemma:recur}.
\item Applying inequalities of Theorem \ref{re}/\ref{th:hooks}.
\end{itemize}

The following lemma follows from Naimark's duality.

\begin{lemma}\label{lemma:recur}
Assume that $L_1=N(\alpha-1).$  Then, $\mathbf L\in\TFF(\alpha, N)$ if and only if  $\mathbf
L'\in\TFF(\tilde \alpha, N(\alpha-1))$ where $\mathbf L'=(L_2\geq\cdots\geq L_k)$ and $1/\alpha+1/\alpha'=1$.
 \end{lemma}

It is easy to see that maximality under the majorization partial order is preserved under these dualities and the lemma above.  Unfortunately, there are several TFF sequences missed by majorization and the recursive generation techniques mentioned above.  These sequences were only found by brute force construction of Littlewood-Richardson tableaux.  The first maximal tight fusion frame sequence missed by recursion is $(4,2,2,2,1)$ where $N=6.$ Hence, it might be of interest to illustrate how to construct a tight fusion frame for this sequence.

\begin{example} \label{ex1a}
Let $N=6$ and $\mathbf L = (4,2,2,2,1)$.
The first step in our construction is identifying a skew Littlewood-Richardson tableaux corresponding to our TFF sequence.

\[
\begin{young}
?1&?1&?1&?1&?1&?1&!!1&!!1&!!1&!!1&!!1\\
?2&?2&?2&?2&?2&?2&!!2&!!2&!!2&??1&??1\\
?3&?3&?3&?3&?3&?3&??1&??1&??1&??2&??2\\
?4&?4&?4&?4&?4&?4&!1&!1&!1&!1&!1\\
!!1&!!2&!!2&??1&??2&!1&!2&!2&!2&!2&!2\\
!!2&??2&??2&??2&!2&???1&???1&???1&???1&???1&???1\\
\end{young}
\]
The above tableaux shows the existence of projections $P_1$, $\ldots$, $P_5$ in $\RR^6$ with
\begin{equation}\label{ex2}
\sum_{i=1}^5 P_i = \frac{11}6 \mathbf I, \quad\rank P_1=4,\  \rank P_2=\rank P_3=\rank P_4=2,\ \rank P_5=1.
\end{equation}
By Theorem \ref{th:hornrec} and Corollary \ref{th:combcharcor}, the tableaux also contains information on the eigenvalues of the intermediate partial sums of projections in \eqref{ex2}.

\begin{center}
\begin{tabular}{|c|c|}\hline
sum of projections & eigenvalue list \\ \hline
$P_1$ & $(1,1,1,1,0,0)$\\[2pt]
 $P_1+P_2$ & $(\frac{11}6,\frac96,1,1,\frac36,\frac16)$\\[2pt]
 $P_1+P_2+P_3$ & $(\frac{11}6,\frac{11}6,\frac{11}6,1,\frac56,\frac46)$\\[2pt]
 $P_1+\ldots+P_4$ & $(\frac{11}6,\frac{11}6,\frac{11}6,\frac{11}6,\frac{11}6,\frac56)$\\[2pt]
 $P_1+\ldots+P_5$ & $(\frac{11}6,\frac{11}6,\frac{11}6,\frac{11}6,\frac{11}6,\frac{11}6)$\\[2pt]
\hline
\end{tabular}
\end{center}

Equipped with this information and a symbolic computation program such as Mathematica we can construct an explicit tight fusion frame in $\RR^6$ associated with the sequence $(4,2,2,2,1)$. The matrix below shows an orthonormal basis (column) vectors for the corresponding ranges of projections $P_i$, $i=1,\ldots,5$.

\[
\left(
\begin{array}{cccc|cc|cc|cc|c}
 1 & 0 & 0 & 0 & \frac{5}{6} & 0 & -\sqrt{\frac{5}{72}} & 0 & \sqrt{\frac{5}{72}} & 0 & 0
\\[2pt]
 0 & 1 & 0 & 0 & 0 & \frac{1}{2} & -\frac{1}{2 \sqrt{2}} & -\frac{1}{3} & -\frac{1}{2 \sqrt{2}} & \frac{1}{3} & \frac{1}{3} \\[2pt]
 0 & 0 & 1 & 0 & 0 & 0 & 0 & \frac{\sqrt{5}}{3} & 0 & \frac{\sqrt{5}}{6} & \frac{\sqrt{5}}{6}
 \\[2pt]
 0 & 0 & 0 & 1 & 0 & 0 & 0 & 0 & 0 & \sqrt{\frac{5}{12}} & -\sqrt{\frac{5}{12}}
 \\[2pt]
 0 & 0 & 0 & 0 & 0 & -\frac{\sqrt{3}}{2} & \frac{1}{2 \sqrt{6}} & -\frac{1}{\sqrt{3}} & \frac{1}{2 \sqrt{6}} & \frac{1}{\sqrt{3}} & \frac{1}{\sqrt{3}}
 \\[2pt]
 0 & 0 & 0 & 0 & -\frac{\sqrt{11}}{6} & 0 & -\sqrt{\frac{55}{72}} & 0 & \sqrt{\frac{55}{72}} & 0 & 0
\end{array}
\right)
\]

A direct calculation shows that: (i) columns are orthonormal to each other in every block, and (ii) rows are orthogonal with norms $\sqrt{11/6}$. This proves the existence of a TFF \eqref{ex2}.
\end{example}

It is worth noting that the Example \ref{ex1a} can not be obtained using the spectral tetris construction (STC). The STC has been recently introduced by Casazza et al. \cite{CFHWZ} who gave an algorithmic way of constructing sparse fusion frames. Among other things, the authors of \cite{CFHWZ} have shown that the ranks $\mathbf L$ of spectral tetris fusion frames must necessarily satisfy $\mathbf L \preccurlyeq \mathbf L'$, where $\mathbf L'$ is a sequence of ranks of the reference fusion frame. Moreover, in the tight case this condition is also sufficient, and hence \cite[Theorem 3.3]{CFHWZ} characterizes possible ranks obtained by the STC in the case when the frame bound $\alpha\ge 2$. Combining this with Naimark's complements, see Theorem \ref{du2}, this yields TFFs also in the case $1<\alpha<2$. In particular, we have $\TFF(11/6,6)=\TFF(11/5,5)$. A direct calculation of the reference fusion frame corresponding to eigenvalues $(11/5,11/5,11/5,11/5,11/5)$ yields a TFF sequence $\mathbf (3,3,3,2)$. This happens to be another maximal element of $\TFF(11/6,6)$ which is not comparable with $(4,2,2,2,1)$ with respect to the majorization relation $\preccurlyeq$. Hence, the above example can not be obtained by the STC even when paired with Naimark's duality.

\bigskip

\begin{center} \sc List of maximal TFF sequences  for $N\leq 9$ and $\alpha
\leq 2$.
\end{center}

{\begin{center}
\begin{tabular}{|c|c|}
\hline
\multicolumn{2}{|c|}{$N=3$} \\
\hline
$\alpha$ & max elements\\ \hline
$1$ & $(3)$\\
 $4/3$ & $(1,1,1,1)$\\
 $5/3$ & $(2,1,1,1)$\\
 $2$ & $(3,3)$\\
\hline
\end{tabular}\hspace{.5in}
\begin{tabular}{|c|c|}
\hline
\multicolumn{2}{|c|}{$N=4$} \\
\hline
$\alpha$ & max elements\\ \hline
$1$ & $(4)$\\
 $5/4$ & $(1,1,1,1,1)$\\
 $6/4$ & $(2,2,2)$\\
 $7/4$ & $(3,1,1,1,1), (2,2,2,1)$\\
 $2$ & $(4,4)$\\
\hline
\end{tabular}\end{center}}

\bigskip

{\begin{center}
\begin{tabular}{|c|c|}
\hline
\multicolumn{2}{|c|}{$N=5$} \\
\hline
$\alpha$ & max elements\\ \hline
$1$ & $(5)$\\
 $6/5$ & $(1,1,1,1,1,1)$\\
 $7/5$ & $(2,2,1,1,1)$\\
 $8/5$ & $(3,2,1,1,1), (2,2,2,2)$\\
 $9/5$ & $(4,1,1,1,1,1), (3,2,2,2)$\\
 $2$ & $(5,5)$\\
 \hline
\end{tabular}\hspace{.1in}
\begin{tabular}{|c|c|}
\hline
\multicolumn{2}{|c|}{$N=6$} \\
\hline
$\alpha$ & max elements\\ \hline
$1$ & $(6)$\\
 $7/6$ & $(1,1,1,1,1,1,1)$\\
 $8/6$ & $(2,2,2,2)$\\
 $9/6$ & $(3,3,3)$\\
 $10/6$ & $(4,2,2,2)$\\
 $11/6$ & $(5,1,1,1,1,1,1),(4,2,2,2,1), (3,3,3,2)$\\
 $2$ & $(6,6)$\\
\hline
\end{tabular}\end{center}}

\bigskip

{\begin{center}
\begin{tabular}{|c|c|}
\hline
\multicolumn{2}{|c|}{$N=7$} \\
\hline
$\alpha$ & max elements\\ \hline
$1$ & $(7)$\\
 $8/7$ & $(1,1,1,1,1,1,1,1)$\\
 $9/7$ & $(2,2,2,1,1,1)$\\
 $10/7$ & $(3,3,1,1,1,1),(3,2,2,2,1)$\\
 $11/7$ & $(4,3,1,1,1,1), (4,2,2,2,1)$\\
 $12/7$ & $(5,2,2,1,1,1),(4,3,3,1,1),(3,3,3,3)$\\
 $13/7$ & $(6,1,1,1,1,1,1,1),(5,2,2,2,2),(4,3,3,3)$\\
 $2$ & $(7,7)$\\
 \hline
\end{tabular}\end{center}}

\bigskip

{\begin{center}
\begin{tabular}{|c|c|}
\hline
\multicolumn{2}{|c|}{$N=8$} \\
\hline
$\alpha$ & max elements\\ \hline
$1$ & $(8)$\\
 $9/8$ & $(1,1,1,1,1,1,1,1,1)$\\
 $10/8$ & $(2,2,2,2,2)$\\
 $11/8$ & $(3,2,2,2,2),(3,3,2,1,1,1)$\\
 $12/8$ & $(4,4,4)$\\
 $13/8$ & $(5,3,2,1,1,1),(5,2,2,2,2),(4,4,2,2,1)$\\
 $14/8$ & $(6,2,2,2,2),(5,3,3,2,1),(4,4,4,2)$\\
 $15/8$ & $(7,1,1,1,1,1,1,1),(6,2,2,2,2,1),(5,3,3,2,2), (4,4,4,3)$\\
 $2$ & $(8,8)$\\
 \hline
\end{tabular}\end{center}}

\bigskip

{\begin{center}
\begin{tabular}{|c|c|}
\hline
\multicolumn{2}{|c|}{$N=9$} \\
\hline
$\alpha$ & max elements\\ \hline
$1$ & $(9)$\\
 $10/9$ & $(1,1,1,1,1,1,1,1,1,1)$\\
 $11/9$ & $(2,2,2,2,1,1,1)$\\
 $12/9$ & $(3,3,3,3)$\\
 $13/9$ & $(4,4,1,1,1,1,1),(4,3,2,2,2),(3,3,3,3,1)$\\
 $14/9$ & $(5,4,1,1,1,1,1),(5,3,2,2,2),(4,3,3,3,1)$\\
 $15/9$ & $(6,3,3,3)$\\
 $16/9$ & $(7,2,2,2,1,1,1),(6,3,3,3,1),(5,4,4,2,1), (4,4,4,4)$\\
 $17/9$ & $(8,1,1,1,1,1,1,1,1,1), (7,2,2,2,2,2), (6,3,3,3,2), (5,4,4,4)$\\
 $2$ & $(9,9)$\\
 \hline
\end{tabular}\end{center}}

\bibliographystyle{abbrv}
\bibliography{refTTFFHC}

\end{document}